\documentclass{amsart}

\usepackage[utf8]{inputenc}
\usepackage{amsmath}
\usepackage{amsfonts,amsthm,amssymb,color,cancel,todonotes,hyperref,stmaryrd}
\usepackage[utf8]{inputenc}
\usepackage[T1]{fontenc}
\usepackage{amsfonts}
\usepackage{comment}
\usepackage{amssymb}
\usepackage[capitalise]{cleveref}
\usepackage[version=4]{mhchem}
\usepackage{stmaryrd}
\usepackage{hyperref}
\hypersetup{colorlinks=true, linkcolor=blue, filecolor=magenta, urlcolor=cyan,}
\usepackage{graphicx} 
\usepackage{todonotes}

\usepackage[nottoc, notlof, notlot]{tocbibind}
\usepackage{graphicx} 
\usepackage{caption}
\usepackage{placeins}
\usepackage{enumitem,kantlipsum}
\usepackage{faktor}

\usepackage[english]{babel}
\usetikzlibrary{arrows}

\newcounter{results}[section] 

\theoremstyle{plain}
\newtheorem{theorem}[results]{Theorem}
\newtheorem{lemma}[results]{Lemma}
\newtheorem{proposition}[results]{Proposition}
\newtheorem{corollary}[results]{Corollary}

\newtheorem*{theorem*}{Theorem}
\newtheorem*{lemma*}{Lemma}
\newtheorem*{proposition*}{Proposition}
\newtheorem*{corollary*}{corollary}
\newtheorem*{exercise*}{Exercise}
\newtheorem*{fact*}{Fact}
\newtheorem*{problem*}{Problem}

\theoremstyle{remark}
\newtheorem{remark}[results]{Remark}

\newtheorem*{remark*}{Remark}
\newtheorem*{question*}{Question}

\theoremstyle{definition}
\newtheorem{definition}[results]{Definition}
\newtheorem{example}[results]{Example}

\newtheorem*{definition*}{Definition}
\newtheorem*{example*}{Example}

\numberwithin{equation}{section}

\newcommand{\N}{\ensuremath{\mathbb N}} 
\newcommand{\R}{\ensuremath{\mathbb R}} 

\newcommand{\eps}{\varepsilon}
\newcommand{\mres}{\mathbin{\vrule height 1.6ex depth 0pt width
0.13ex\vrule height 0.13ex depth 0pt width 1.3ex}}
\newcommand{\de}{\ensuremath{\,\mathrm d}}

\usepackage[backend=bibtex,giveninits=true, doi=false, style=alphabetic,bibencoding=utf8,language=auto,autolang=other,maxnames=10, maxalphanames=10]{biblatex}
\hypersetup{
colorlinks=true,
linkcolor={blue},
urlcolor={red},
citecolor={blue}
}

\renewbibmacro{in:}{} 
\DeclareFieldFormat[article]{pages}{#1} 
\DeclareFieldFormat[book]{pages}{#1}
\DeclareFieldFormat[article]{title}{#1} 
\DeclareFieldFormat[book]{title}{\textit{#1}}

\renewbibmacro*{volume+number+eid}{ 
  \printfield{volume}
  \setunit*{\addcomma \addnbspace}
  \printfield{number}
  \setunit{\addcomma\space}
  \printfield{eid}}
  
\renewbibmacro*{publisher+location+date}{
  \printlist{publisher}
  \iflistundef{location}
    {\setunit*{\addcomma\space}}
    {\setunit*{\addcolon\space}}
  \printlist{location}
  \setunit*{\addcomma\space}
  \usebibmacro{date}
  \newunit}

\title[Superposition principle of 1-d local currents]{The superposition principle for \\ local 1-dimensional currents}
\author{L. Ambrosio}
\address{Scuola Normale Superiore, Pisa}
\email{luigi.ambrosio@sns.it}
\author{F. Renzi}
\address{Scuola Normale Superiore, Pisa}
\email{federico.renzi@sns.it}
\author{F. Vitillaro}
\address{Scuola Normale Superiore, Pisa}
\email{federico.vitillaro@sns.it}
\date{March 2025}

\begin{document}

\maketitle

\begin{abstract}
We prove that every one-dimensional locally normal metric current, intended in the sense of U. Lang and S. Wenger, admits a nice integral representation through currents associated to (possibly unbounded) curves with locally finite length, generalizing the result shown by E. Paolini and E. Stepanov in the special case of Ambrosio-Kirchheim normal currents. Our result holds in Polish spaces, or more generally in complete metric spaces for 1-currents with tight support.
\end{abstract}

\section{Introduction} Superposition principles are very useful tools in Analysis and Probability, as they allow to factorize a problem into lower dimensional ones.
Classical examples are the coarea formula for $BV$ functions, the slicing theory of currents and the Optimal Transport problem with cost=distance (in the latter
case, the decomposition being given by transport rays). In the context of 1-dimensional currents, it is particularly interesting to extend the classical S. Smirnov's \cite{SM} superposition principle even to the case when the ambient space is a metric space, having in mind the well-posedness of ODE's in metric measure spaces 
(see \cite{AT2}, \cite{ST}) or the theory of Optimal Transport, see \cite{PS1} and \cite{PS2}.

In this paper we prove that every one-dimensional locally normal metric current admits a decomposition, thus extending the results of E. Paolini and E. Stepanov in \cite{PS1} and \cite{PS2} for Ambrosio-Kirchheim normal currents.
A notion of local metric current was first introduced by U. Lang in \cite{UL1}. Later, he gave, alongside S. Wenger, a different, non equivalent definition, in  \cite{UL2}. We will work with the latter (recalled in \cref{sectioncurrents} below). Roughly speaking, this notion generalizes Ambrosio-Kirchheim metric currents by requiring finite mass only on bounded sets. In \cref{section-acyc} of this paper, we provide the natural concepts of ``subcurrent'' and ``acyclic current'' for this class of objects.
As it is customary in this field, we will work with the requirement that the mass measures of currents are tight, which is not restrictive, for instance, when the ambient metric space is Polish (see also Remark~\ref{rem-tightness}).

Let us define, preliminarly, the classes of curves we will be dealing with. As in \cite{PS1} and \cite{PS2} we consider compact
curves with finite length (which always admit a Lipschitz parameterization) and the distance induced by the sup norm 
when we consider the invariance under reparameterization:

\begin{definition}\label{defTheta}
Let $(E,d)$ be a complete metric space. We equip the set of Lipschitz curves $\theta:[0,1] \rightarrow E$ with the distance
\begin{equation} d_\Theta(\theta_1,\theta_2):=\inf\left\{
\max_{t\in [0,1]}
d(\theta_1(t),\theta_2(\phi(t)))\,:\,\phi\colon [0,1]\to [0,1] \mbox{ bijective increasing}
\right\},   
\end{equation}
and call two curves $\theta_i:[0,1]\to E$, $i=1,2$, equivalent, if $d_\Theta(\theta_1,\theta_2)=0$. 

The set of equivalence classes of Lipschitz curves equipped with the distance $d_\Theta$ will be denoted by $\Theta(E)$.
\end{definition}

\begin{remark}
$(\Theta(E), d_\Theta)$ can be written as a countable union of complete subspaces of the form $\Theta_L=\{[\theta] \in \Theta(E): \ell(\theta) \leq L \}$, $L \in \N$. For example, if $E$ is separable (and then $\Theta(E)$ is separable as well), this gives that any finite Borel measure $\eta$ on $\Theta(E)$ is tight.    
\end{remark}

A notable distinction from the classical case is that, in order to maintain the transport property of the representation of acyclic currents (which comes from (\ref{boundaryeq}) below), highlighted in \cite{PS1}, it is necessary to consider curves which are not necessarily of bounded length. This can be easily understood by looking at the basic example of the current associated to an unbounded simple oriented curve with a single boundary point (e.g. a half-line in $\R^d$), where, as it is somehow natural, the decomposition must be given by the curve itself. Clearly, the transport property has to be meant in a wider sense, involving the possibility of exchanging mass ``with infinity'' (see \cref{transport-remark}). Therefore in the context of local currents, it is more convenient to consider, instead, open 
ended curves with locally finite length in $E$, according to the following definition:

\begin{definition}\label{DefGamma}
Consider
$$X:=\bigl\{ \beta \in \operatorname{Lip}_{\rm loc}((0,1); E)\mid \ell(\beta \mres U) < \infty\,\,\text{for any bounded open set $U\subseteq E$}
\bigr\},$$
where $\ell(\beta\mres U)=\int_{\beta^{-1}(U)}|\dot\beta|\de t$ and $|\dot\beta|$ is the metric derivative of $\beta$.

We define the set $\Gamma(E)$ as the quotient $\faktor{X}{\sim}$, where $\beta_1 \sim \beta_2$ iff there exists $\phi: (0,1) \rightarrow (0,1)$ increasing and bijective s.t. $\beta_1=\beta_2 \circ \phi$.
\end{definition}

Notice that the difference between $\Theta(E)$ and $\Gamma(E)$ is essentially due to the requirement of locally finite length,
since any open ended curve with finite length has a unique extension to a curve in $\Theta(E)$. Actually, the local finiteness of the length of $\beta \in \Gamma(E)$ gives the following dichotomy: either $\lim_{t \to 0} \beta(t)$ exists, or $\lim_{t \to 0} d(\beta(t), \bar{x})=\infty$, for any $\bar{x} \in E$ (same for $t \to 1$). In addition, since all concepts
we will be dealing with are invariant under reparameterization, we will occasionally consider other open of closed intervals
as domains of our curves.

In Section~\ref{sectioncurves} we introduce, mainly for measure-theoretic purposes, the topology $\tau_\Gamma$ induced by the map
$\beta\mapsto T_\beta$ in \eqref{defTbeta}, canonically associating to $\beta\in\Gamma(E)$ the local 1-dimensional current in $E$. In this
connection, notice that this topology is not Hausdorff in $\Gamma(E)$ since $T_\beta=T_{\tilde{\beta}}$ does not imply that
$\beta$ is a reparameterization of $\tilde{\beta}$ (for example, concatenating $\beta$ with two copies with opposite orientation of a curve $\alpha$ with $\alpha(0)=\beta(1)$ does not change $T_\beta$).

Our main result is the following (here $\mathbf{N}_{1,b}(E)$ denotes the class of local and normal $1$-dimensional currents in $E$,
see Definition~\ref{defAKcurrent} and Definition~\ref{defnormal}).

\begin{theorem}\label{MT}
Let $(E,d)$ be a complete metric space and let $T \in \mathbf{N}_{1,b}(E)$. Then there exists 
a positive Borel measure $\eta$ over $\Gamma(E)$ such that
\begin{align}
&T=\int_{\Gamma(E)} T_\gamma \de\eta(\gamma), \label{decomp} \\
&\|T\|=\int_{\Gamma(E)} \|T_\gamma\| \de\eta(\gamma). \label{masseq}
\end{align}
Moreover, $T$ can be decomposed as $T=A+C$, where $C$ is a  cycle of $T$ and $A$ is acyclic.
Finally, denoting by $\eta_A$  the measure associated to $A$  as in \eqref{decomp} and \eqref{masseq}, we have 
that $\eta_A$-almost every curve is injective, and that
\begin{equation}
 \|\partial T\|=\|\partial A\|=\int_{\Gamma(E)} \|\partial T_\gamma\| \de\eta_A(\gamma).\label{boundaryeq}   
\end{equation}
\end{theorem}

In this generality, \cref{MT} represents a new result even in the Euclidean setting. However, our proof-technique is intrinsically metric, in the sense that it relies on the validity of the statement for normal currents in general metric spaces, due to Paolini and Stepanov, even when our space $E$ is Euclidean. Indeed, we exploit a simple trick which consists in 
using a $\|T\|$-dependent ``conformally'' modified distance $\delta$ (see \cref{conformal-sec}) of the distance $d$, turning $T$ into a current $\overline{T}$ of finite mass with respect to $\delta$ in the completion $\overline{E}$ of $E$ with respect to $\delta$. Therefore, when the mass of $\partial T$ is assumed to be already finite, we reduce ourselves to an application of Paolini and Stepanov's result, provided we are able to show that still $\partial\overline{T}$ has finite mass (this is not
obvious, see Example~\ref{exnofinitemass}, but it can be obtained at the expense of embedding isometrically our spaces into $\ell^\infty$). 
We recall here Paolini and Stepanov's result for further use (cfr. \cite[Theorem 5.1]{PS1} and \cite[Corollary 3.3]{PS2}).

\begin{theorem}\label{PS}
Let $T \in \mathbf{N}_{1}(E)$. Then there exists a finite positive Borel measure $\bar{\eta}$ over $\Theta(E)$ with total mass $\bar{\eta}(\Theta(E)) \leq \mathbb{M}(T)+\mathbb{M}(\partial T)$ such that
\begin{align}
&T(\omega)  =\int_{\Theta(E)} T_\theta(\omega) \de \bar{\eta}(\theta) \quad \forall\,\omega =f\de\pi\in \operatorname{Lip}_{\mathrm{b}}(E) \times \operatorname{Lip}(E),\label{rappresentazione}\\
&\label{massa}
\mathbb{M}(T)  =\int_{\Theta(E)} \mathbb{M}(T_\theta) \de \bar{\eta}(\theta)=\int_{\Theta(E)} \ell(\theta) \de \bar{\eta}(\theta), 
\end{align}
with $\bar{\eta}$-a.e. $\theta \in \Theta(E)$ contained in $\operatorname{supp}(T)$.\\
Furthermore, $T$ can be decomposed into a sum $T=A+C$ where $\partial C=0$, $C\leq T$ and  $A$ is acyclic.
In addition,
\begin{itemize}
\item[(i)] The measure $\eta_A$ obtained by applying the result on $A$ as in (\ref{rappresentazione}) and \eqref{massa}
satisfies $\eta_A (1)=(\partial T)^+$, $\eta_A (0)=(\partial T)^-$ and
$\eta_A$-a.e. curve in $\Theta(E)$ is injective; in particular
\begin{equation}
 \mathbb{M}(\partial T) =\mathbb{M}(\partial A)=\int_{\Theta(E)} \mathbb{M}(\partial T_\theta) \de \eta_A(\theta). 
\end{equation}
\item[(ii)] the measure $\eta_C$ obtained by applying the result on $C$ as in (\ref{rappresentazione}) and \eqref{massa} satisfies $\eta_{C}(0)=\eta_{C}(1)=\|C\|$ and
$$\mathbb{M}(T_\theta)=\ell(\theta)=1\qquad\text{for $\eta_C$-a.e. curve in $\Theta(E)$.}$$
\end{itemize}
\end{theorem}
   
We remark that the main technical difficulties faced in order to state and prove \cref{MT} arise from the issue of infinite mass for the current $T$, which was also the original motivation for this work, see the forthcoming paper \cite{AILP2} where this problem is raised. On the other hand, dropping the condition $\mathbb{M}(\partial T)<\infty$, which is done in \cref{section-acyc}, is relatively easier. Indeed, first in \cref{decompcycacyc} we prove the decomposition $T=A+C$ with $C$ (maximal) subcycle of $T$ and $A$ acyclic, for any $m$-dimensional local current $T$. Then, focusing on the acyclic part in the case $m=1$, we show that it can be written as a countable sum of subcurrents with finite boundary, with no cancellations even at the level of boundaries, see \cref{infinite-boundary}, a result which may also be of independent interest. This is done almost by hand, using the local finiteness of $\| \partial T\|$ and applying \cref{PS} on the restriction of $T$ to balls in a clever way. \\
To conclude, note that our work does not provide a result of decomposition of cycles in elementary subcycles, as it was done in the case of finite mass already in \cite{SM} and again in \cite{PS2} in the metric setting, involving averages on ``almost periodic paths'' or ``solenoids'' (see \cite[Section 4]{PS2} for precise definitions and statements in this sense) .
Solenoids cannot in general be decomposed into closed curves, and so in some sense they are the most elementary 1-dimensional cycles which can be used for this purpose. \\
In our case, the matter is that a decomposition in solenoids is not preserved by reversing the metric construction that we perform to reduce ourselves to the case $\mathbb{M}(T)<\infty$, and so its existence does not seem to follow from our approach.\\

{\bf Acknowledgments.} The authors have been supported by the MIUR-PRIN 202244A7YL project "Gradient Flows and Non-Smooth Geometric Structures with Applications to Optimization and Machine Learning".

\section{Preliminaries}

In this section, $(E,d)$ will denote a generic metric space. We denote by $\operatorname{Lip}_{\mathrm{b}}(E)$ and $\operatorname{Lip}_{\mathrm{bs}}(E)$ the spaces of bounded Lipschitz functions and of Lipschitz functions with bounded support in $E$ respectively. We denote by
${\mathcal B}(E)$ the $\sigma$-algebra of Borel sets and by ${\mathcal B}_b(E)\subseteq {\mathcal B}(E)$ 
the class of bounded Borel sets.  Analogously, we denote by $\mathcal{B}^\infty(E)$ and
$\mathcal{B}_{\mathrm{bs}}^\infty(E)$ the class of bounded Borel functions and bounded 
Borel functions with bounded support, respectively. 

In the sequel, $\mathbf{M}_0(E)$ denotes the class of Borel $\sigma$-additive measures 
$\mu:{\mathcal B}(E)\to\R$ with finite total variation concentrated on a  $\sigma$-compact set 
and by $\mathbf{M}_{0,b}(E)$ the class of set functions $\mu:{\mathcal B}_b(E)$
such that $B\mapsto\mu(U\cap B)$ belongs to $\mathbf{M}_0(E)$ for any bounded Borel set $U$.
Clearly any nonnegative $\mu\in\mathbf{M}_{0,b}(E)$ can be canonically and monotonically extended 
to a $\sigma$-additive set function, for which we use the same notation $\mu$, defined on the whole of ${\mathcal B}(E)$, 
which then belongs to $\mathbf{M}_0(E)$ iff $\mu(E)<\infty$.

\begin{remark}\label{rem-tightness}
Note that the ``tightness'' requirement in the previous definition is automatically satisfied if $E$ enjoys topological properties ensuring that every finite nonnegative Borel measure is tight.
This is true when $(E,d)$ is a Polish space or if we assume (consistently with the Zermelo-Fraenkel set theory, as done
also in \cite{AK}) that the density character of every metric space considered is an Ulam number (see \cite[Chapters 1 and 7]{B}).
\end{remark}

\subsection{(Local) metric currents}\label{sectioncurrents}

We now introduce the notions of metric currents that we will need, starting with the case of $0$-dimensional currents. 

\begin{definition} [$0$-dimensional metric currents]
We say that a linear functional $T: \operatorname{Lip}_{\mathrm{bs}}(E)\rightarrow\R$ is a $0$-dimensional current with locally finite mass
if there exists a nonnegative $\mu\in\mathbf{M}_{0,b}(E)$ such that
$$
|T(f)|\leq \int_E|f| \de\mu\qquad\forall f\in \operatorname{Lip}_{\mathrm{bs}}(E).
$$
It is easily seen that the class of nonnegative $\mu\in\mathbf{M}_{0,b}(E)$ satisfying the inequality above
is a lattice in the class of nonnegative $\mu\in\mathbf{M}_{0,b}(E)$. The least one is called mass of $T$ and denoted
by $\|T\|$.

We say that $T$ has finite mass if the inequality above holds for some $\mu\in\mathbf{M}_{0}(E)$. 
\end{definition}

Clearly any $0$-dimensional current with locally finite mass is representable by integration with respect to a
measure $\mu_T\in \mathbf{M}_{0,b}(E)$, namely $T(f)=\int_E f \de\mu_T$, which shows that we can 
canonically identify the class of 
$0$-dimensional metric currents with locally finite mass with $\mathbf{M}_{0,b}(E)$. In the case of
$0$-dimensional currents $T$ with finite mass, the canonical identification is with $\mathbf{M}_0(E)$, so that
in particular $T(f)$ is well defined when $f\in\operatorname{Lip}_b(E)$ (or even when $f$ is a bounded Borel function).
 
Now, in the same vein, we consider the case $m>0$, where the main difference between the local and the non local
case consists in the replacement of $\operatorname{Lip}_{\mathrm{bs}}(E)$ with
$\operatorname{Lip}_b(E)$ in the first argument of $T$, see \cite{AK} and \cite{UL2}.

\begin{definition}[Metric currents with finite and locally finite mass]\label{defAKcurrent}
\noindent A functional $T(f\,\de\pi_1\wedge\cdots\wedge\de\pi_m)$ defined on
$$
\operatorname{Lip}_{\mathrm{bs}}(E) \times\left[\operatorname{Lip}(E)\right]^m \rightarrow \mathbb{R}
\qquad m \geq 1
$$
is called an $m$-dimensional metric current with locally finite mass on $E$ if the following properties hold:\\
\begin{enumerate}
\item\label{point1defmass} $T$ is multilinear;\\
\item\label{point2defmass} $T$ is continuous in the following sense: if $\pi^j_i, \pi_i \in \operatorname{Lip}(E)$,
$\pi_{i}^{j}\rightarrow\pi_{i}$ as $j\to\infty$ pointwise in $E$ for all $i=1,\ldots,m$ and $\sup_{i, j} \operatorname{Lip}(\pi_{i}^{j})<\infty$, then
$$
\lim_{j\to\infty}T(f\,\de\pi_1^j\wedge\cdots\wedge\de\pi_m^j) = 
T\left(f\,\de\pi_1\wedge\cdots\wedge\de\pi_m\right)\qquad\forall f\in \operatorname{Lip}_{\mathrm{bs}}(E);
$$
\item\label{point3defmass} $T(f\,\de\pi_1\wedge\cdots\de\pi_m)=0$ whenever 
$\pi_{i}\vert_{\operatorname{supp}f}=0$ for some $i\in\{1,\ldots,m\}$;
\item\label{point4defmass} there exists a nonnegative $\mu\in\mathbf{M}_{0,b}(E)$ such that
\begin{equation}\label{inequalitAKcurrents}
\left|T(f\, \de\pi_1\wedge\cdots\wedge\de\pi_m)\right| \leq \prod_{i=1}^{m} \operatorname{Lip}(\pi_{i}) \int_{E}|f| \de\mu
\end{equation}
holds for all $(f, \pi_1, \ldots,\pi_m)\in\operatorname{Lip}_{\mathrm{bs}}(E) \times\left[\operatorname{Lip}(E)\right]^m$.
\end{enumerate}

We say that $T$ has finite mass if \eqref{inequalitAKcurrents} holds for some $\mu\in \mathbf{M}_0(E)$.
We denote by $\mathbf{M}_{m}(E)$, $\mathbf{M}_{m,b}(E)$, the vector spaces of $m$-dimensional
metric currents with finite mass and locally finite mass, respectively, so that
$\mathbf{M}_{m}(E)\subseteq\mathbf{M}_{m,b}(E)$.
\end{definition} 

\begin{remark}[Extensions of currents]
Using the density of Lipschitz functions with bounded support in $L^1(E,\|T\|)$, currents with locally finite mass
can be canonically extended to $\mathcal{B}_{\mathrm{bs}}^{\infty}(E)\times\left[\operatorname{Lip}(E)\right]^m$
retaining properties (1)-(4). Analogously, currents with finite mass can be canonically extended to 
$\mathcal{B}^\infty(E)\times\left[\operatorname{Lip}(E)\right]^m$
retaining properties (1)-(4).  In particular, the possibility to extend currents with finite mass to
$\operatorname{Lip}_{\mathrm{b}}(E) \times\left[\operatorname{Lip}(E)\right]^{m}$ grants that
our presentation is consistent with the axiomatization in \cite{AK}, see also Remark~\ref{rem:extensions} below.
\end{remark}

\begin{definition}[Mass and support of currents]
As in the case $m=0$,  for $T\in\mathbf{M}_{m,b}(E)$, we call mass and denote $\|T\|$ the least nonnegative
measure $\mu\in\mathbf{M}_{0,b}(E)$ satisfying \eqref{inequalitAKcurrents}. When 
$T\in\mathbf{M}_{m}(E)$, the total mass of $T$ is defined by $\mathbb{M}(T):=\|T\|(E)$ and one can 
easily show that $(\mathbf{M}_{m}(E),\mathbb{M})$ is a  Banach space.

For $T\in \mathbf{M}_{m,b}(E)$, the support of $T$ is defined as the closed set
$$
\operatorname{supp} T:=\operatorname{supp}\|T\|=\{z \in E:\ \|T\|(B_r(z))>0 \quad \forall r>0\}.
$$
If $\Sigma$ is any $\sigma$-compact set with $\|T\|(E \setminus \Sigma)=0$, then $\Sigma \setminus \operatorname{supp} T$ is contained in the union of countably many $\|T\|$-negligible open balls, thus
$\|T\|(E \setminus \operatorname{supp} T)=0$, i.e. $\|T\|$ is concentrated on $\operatorname{supp} T$.
\end{definition}

\begin{remark}\label{rem_dualmass} Arguing as in \cite{AK} and \cite{UL2} one can prove that on open
sets $A$ the mass measure $\|T\|(A)$ can be computed as follows:
$$
\sup \sum_{\lambda\in\Lambda}T(f_\lambda\,\de\pi_{\lambda,1}\wedge\cdots\wedge\de\pi_{\lambda,m})
$$
where the supremum runs among all finite families $\{(f_\lambda,\pi_{\lambda,1},\ldots,\pi_{\lambda,m})\}_{\lambda\in\Lambda}$, with $\pi_{\lambda,i}\in\operatorname{Lip}_1(E)$ and $\sum_{\lambda\in\Lambda}|f_\lambda|\leq\chi_A$.
\end{remark}

Along with the strong notion of convergence of currents given by the mass norm $\mathbb{M}$, it is often useful to consider the usual weak convergence.

\begin{definition}[Weak convergence]
 Let $T_n, \, T \in \mathbf{M}_{m,b}(E)$, we say that $T_n$ is weakly convergent to $T$, and write $T_n \rightharpoonup T$, if
 $$\lim_{n \rightarrow \infty} T_n(f \de\pi_1 \wedge \ldots \de\pi_m)=T(f \de\pi_1 \wedge \ldots \de\pi_m)$$
 for all $(f, \pi_1, \ldots, \pi_m) \in \operatorname{Lip}_{\mathrm{bs}}(E) \times\left[\operatorname{Lip}(E)\right]^m$.
\end{definition}

Note that, by \cref{rem_dualmass}, $\|T\|(A)$ is lower semicontinuous under weak convergence for any open set $A\subseteq E$. 
Also, we introduce the push-forward operator in the class of local metric currents; in the class of Ambrosio-Kirchheim (i.e. currents with finite mass) the condition that the preimage under $\varphi$ of bounded sets is bounded, which grants that
$f\circ\varphi\in \operatorname{Lip}_{\mathrm{bs}}(E)$, is not needed.

\begin{definition}[Push-forward]
Let $T \in \mathbf{M}_{m,b}(E)$, let $(E',d')$ be another metric space and let $\varphi \in \operatorname{Lip}(E, E')$ be such that $\varphi^{-1}(B)$ is bounded for any bounded set $B \subseteq E'$. We define a function
$\varphi_\#T: \operatorname{Lip}_{\mathrm{bs}}(E') \times
\left[\operatorname{Lip}(E')\right]^m \to \R$ by
$$\varphi_\#T(f\, \de\pi_1\wedge\cdots\wedge\pi_m):=
T(f\circ\varphi\, \de(\pi_1\circ\varphi)\wedge\cdots\wedge\de(\pi_m\circ\varphi))$$
and call it the push-forward under $\varphi$ of $T$. Note that $\varphi_\#T \in \mathbf{M}_{m,b}(E')$ 
with $\|\varphi_\#T\|\leq\bigl[\operatorname{Lip}(\varphi)\bigr]^m\varphi_\#\|T\|$.
\end{definition}

\begin{definition}[Restriction]
For $T\in\mathbf{M}_{m,b}(E)$, $g\in \mathcal{B}^{\infty}(E)$, $\tau\in \left[\operatorname{Lip}(E)\right]^{k}$, $0\leq k\leq m$,  we denote by  $T \, \mres (g, \tau)\in \mathbf{M}_{m-k,b}(E)$  the restriction of $T$ to $(g,\tau)$, defined by 
$$T\mres (g\,\de\tau_1\wedge\cdots\wedge\de\tau_k) (f\,\de\pi_1\wedge\cdots\wedge\de\pi_{m-k}):=T(fg\,\de\tau_1\wedge\cdots\wedge\de\tau_k\wedge\de\pi_1\wedge\cdots\wedge\de\pi_{m-k}).$$
In the case $k=0$, notice that $\|T\mres g\|=|g|\|T\|$, 
so that $\mathbb{M}(T\mres g)=\int_E |g|\de\|T\|<\infty$ if $g\in \mathcal{B}_{\mathrm{bs}}^{\infty}(E)$. In that case, we will use the simplified notation $gT:=T \mres g$. If $B\subseteq E$ Borel, we will also denote $T\mres B \,:=T\mres\chi_B=\chi_B T$. 
\end{definition}

\begin{remark} \label{rem:extensions} 
Using multiplication by elements in $\mathcal{B}_{\mathrm{bs}}^{\infty}(E)$, several properties established for currents with finite mass can be immediately extended to currents with locally finite mass, in particular 
the alternating property, which justifies the ``wedge'' notation we used from the very beginning (alternatively,
one can borrow results from \cite{UL2}). 
\end{remark}

\begin{definition}[Boundary]\label{def-boundary}
Let $T \in \mathbf{M}_{m,b} (E)$. We define its boundary 
$\partial T: \operatorname{Lip}_{\mathrm{bs}}(E)\times\left[\operatorname{Lip}(E)\right]^{m-1} \to \R$ by
$$\partial T(f \,\de\pi_1\wedge\cdots\wedge\de\pi_{m-1}):=T(\sigma\, \de f\wedge \de\pi_1\wedge\cdots\wedge\de\pi_{m-1}),$$
where $\sigma \in \operatorname{Lip}_{\mathrm{bs}}(E)$ is any function satisfying 
$\sigma|_{\operatorname{supp} f} \equiv 1$.
\end{definition}
Notice that the definition above is well-posed by the locality property (\ref{point3defmass}) of metric currents, as different choices of $\sigma$ do not change the resulting current.
Furthermore, we have the following Leibniz rule (see e.g. \cite[Eq. (11)]{UL2}):
\begin{equation}\label{1Leib}
    \partial (T\mres \rho)=(\partial T)\mres\rho+T\mres\de\rho\qquad \forall\,T \in \mathbf{M}_{m,b} (E),\, \rho\in\operatorname{Lip}(E).
\end{equation}

In general $\partial T$ need not have finite or locally finite mass. This motivates the next definition.

\begin{definition}[Normal and locally normal currents]\label{defnormal}
We denote by $\mathbf{N}_m(E)$ the class of all $T \in \mathbf{M}_m(E)$ and $\partial T \in \mathbf{M}_{m-1}(E)$. Elements of $\mathbf{N}_1(E)$ are called normal $m$-dimensional currents. The class
 $\mathbf{N}_{m,b}(E)\subseteq \mathbf{M}_{m,b}(E)$ of locally normal $m$-dimensional
 currents is defined analogously, by requiring 
 $\partial T\in\mathbf{M}_{m-1,b}(E)$.
\end{definition}

\begin{remark}
 It is easy to see that weak convergence of normal currents is stable under the operation of boundary.
\end{remark}

In the class of normal metric currents we have the following compactness theorem (cfr. \cite[Theorem 5.2]{AK}).

\begin{theorem}[Compactness]\label{comp-theo}
Let $(T_n) \subseteq \mathbf{N}_m(E)$ be such that $\sup_n (\mathbb{M}(T_n)+\mathbb{M}(\partial T_n)) < \infty$, 
satisfying the following equi-tightness condition: for every $\eps>0$ there exists a compact set $K_\eps \subseteq E$ such that
$$\|T_n\|(E \setminus K_\eps)+\|\partial T_n\|(E \setminus K_\eps) < \eps\quad \forall n \in \N.$$
Then there exists a subsequence $(T_{n(k)})$ weakly convergent to $T \in \mathbf{N}_m(E)$.
\end{theorem}

Using this reasult we can easily infer a suitable variant, holding for equi-bounded and equi-tight sequences of local currents.

\begin{corollary}[Compactness for local currents]\label{comp-theo-loc}
Let $(T_n) \subseteq \mathbf{N}_{m,b}(E)$ be such that for any $A \subseteq E$ open and bounded the following two conditions hold:
\begin{enumerate}
\item $\sup_n (\|T_n\|(A)+\| \partial T_n\|(A)) < \infty$;
\item for every  $\eps>0$ there exists a compact set $K_\eps \subseteq A$ such that
$$\|T_n\|(A \setminus K_\eps)+\|\partial T_n\|(A \setminus K_\eps) < \eps\quad \forall n \in \N.$$
\end{enumerate}
Then there exists a subsequence $(T_{n(k)})$ weakly convergent to $T \in \mathbf{N}_{m, b}(E)$.
\end{corollary}

\begin{proof}
Let $\chi: \R \rightarrow [0,1]$ be a cut-off function s.t. $\chi \equiv 1$ on $(-\infty, 0]$, $\chi \equiv 0$ on $[1,+\infty)$, $\chi$ 1-Lipschitz. Fix $\bar{x} \in E$ and set, for every $R \in \N$, $\chi_R: E \rightarrow [0,1]$, $\chi_R(x):=\chi(d(x, \bar{x})-R)$. We study the sequence $(\chi_R T_n)$, with $R$ fixed. If $B$ is any Borel set, exploiting \eqref{1Leib} and $\| T_n \mres \de \chi_R\| \leq \operatorname{Lip}(\chi_R) \|T_n\|$ we obtain
\begin{equation}\label{stima chiR} \|\chi_R T_n\|(B)+\| \partial(\chi_R T_n)\|(B) \leq 2\|T_n\|(B \cap \overline{B}_{R+1}(\bar{x}))+\|\partial T_n\|(B \cap \overline{B}_{R+1}(\bar{x})).\end{equation}
Then, combining (\ref{stima chiR}) with the hypotheses on $(T_n)$, we are able to apply \cref{comp-theo} to $(\chi_R T_n)$. By a diagonal argument, there exists a subsequence $(T_{n(k)})$ such that $\chi_R T_{n(k)} \rightharpoonup T^R$ for some $T^R \in \mathbf{N}_m(E)$, for all $R \in \N$. We then define a local $m$-dimensional current $T$ by
$$T(f \de \pi_1 \wedge \cdots \wedge d\pi_m):=T^R(f \de \pi_1 \wedge \cdots \wedge \de \pi_m), \quad \operatorname{supp}(f) \subseteq B_R(\bar{x}),$$
where $(f, \pi) \in \operatorname{Lip}_{\mathrm{bs}}(E) \times[\operatorname{Lip}(E)]^m$. Note that, independently from the choice of $R$, we get (as $\chi_R f=f$)
$$T(f \de \pi_1 \wedge \cdots \wedge d\pi_m)=\lim_{k \rightarrow \infty} T_{n(k)} (f \de \pi_1 \wedge \cdots \wedge d\pi_m),$$
which proves both that $T$ is well-defined and the convergence $T_{n(k)} \rightharpoonup T$.
\end{proof}

Later in this work we will need the following estimate, stated and proved for simplicity only in the $1$-dimensional case.
\begin{proposition}\label{deriv-est}
   Let $T \in \mathbf{N}_{1,b}(E)$, $x_0\in E$.  Then $T \mres \overline{B}_r(x_0)\in \mathbf{N}_1(E)$ and
   \begin{equation}\label{eq:blocal}
\| \partial(T \mres \overline{B}_{r}(x_0))\|(\partial B_r(x_0)) \leq \frac \de{\de r} \|T\|(\overline{B}_r(x_0))
   \end{equation}
 for a.e.\, $r>0$.
\end{proposition}
\begin{proof}
 By the local finiteness of the boundary mass, $\|\partial T\|(\partial B_{r}(x_{0}))=0$ for almost all $r>0$.\\
 Furthermore, the function $\alpha(r)=\|T\|(\overline{B}_r(x_0)) $ is increasing, hence, it is almost everywhere differentiable. Fix a radius $r>0$ satisfying all these conditions.
Note that $T \mres {\rho_{n}}\stackrel{n\to\infty}{\longrightarrow} T \mres \overline{B}_{r}(x_{0})$ weakly, where 
$$
\rho_{n}(x)=\begin{cases}
1 & \text{if } x\,\in\,B_{r}(x_{0}) \\
1-n(d(x,x_{0})-r)  & \text{if } x\,\in\,B_{r+\frac{1}{n}}(x_{0})\setminus B_{r}(x_{0})\\
0 & \text{else.} 
\end{cases}
$$

\noindent Based on the conditions on
$r$, the semicontinuity of mass with respect to weak convergence, \cref{1Leib} and the fact that
$\operatorname{Lip}(\rho_n)\leq n$, we can conclude that
 
 \begin{align*}
\| \partial(T \mres \overline{B}_r(x_{0}))\|(E) &\leq\liminf_{n \to \infty} \| \partial(T \mres \rho_n )\|(E) \\
&\leq \liminf_{n \to \infty} \| (\partial T )\mres \rho_n \|(E) + \| T \mres(\sigma\de\rho_n) \|(E) \\
&=\|\partial T\|(\overline{B}_r(x_0))+\liminf_{n \to \infty}  \left\| T \mres\left(\sigma\de\rho_n\right)\right\| (E )\\
&=\|\partial T\|(\overline{B}_r(x_0))+\liminf_{n \to \infty}  \left\| T \mres\left(\chi_n\de\rho_n\right)\right\| (E )\\
&=\|\partial T\|(\overline{B}_r(x_0))+ \frac{\de}{\de r} \|T\|(\overline{B}_r (x_{0})). 
\end{align*}
with $\sigma\in \mathcal{B}^\infty_{bs}(E)$ identically equal to 1 on $B_{2r}(x_0)$ and $\chi_n$ the characteristic function of $\overline{B}_{r+1/n}(x_0)\setminus B_r(x_0)$.
This proves that $T \mres \overline{B}_r(x_0)\in \mathbf{N}_1(E)$. By subtracting 
$\|\partial(T \mres \overline{B}_r(x_{0}))\|(B_r(x_0))=\|\partial T\|(B_r(x_0))$ from both sides, we obtain \eqref{eq:blocal}.
\end{proof}

\subsection{Curves in metric spaces} We introduce a few basic definitions about metric curves, i.e. continuous maps $\gamma: J \subseteq \R \rightarrow E$, where $J$ is an interval  (possibly unbounded, not necessarily closed). Any restriction of $\gamma$ to a subinterval $\tilde{J}\subseteq J$ will be called
a subcurve of $\gamma$. A curve $\gamma:J \subseteq\mathbb{R}\to E$ is called an arc if it is injective.

Recall that $\gamma \in \operatorname{AC}([a,b]; E)$ if 
\begin{equation}\label{defac}
d(\gamma(s),\gamma(t))\leq\int_s^t g(r) \de r\qquad\text{whenever $a\leq s\leq t\leq b$}
\end{equation}
for some $g\in L^1(a,b)$, so that in particular Lipschitz maps are absolutely continuous. For the following result, see for example \cite[Theorem 9.2]{ABS}.

\begin{definition}[Metric derivative]
We define the metric derivative of $\gamma$ at the point $t \in \mathring{J}$  as the limit
$$\lim_{h \rightarrow 0} \frac{d(\gamma(t+h), \gamma(t))}{|h|}$$
whenever it exists and, in this case, we denote it by $|\dot{\gamma}|(t)$.
\end{definition}

\begin{theorem}
For any $\gamma \in \operatorname{AC}([a,b]; E)$ the metric derivative $|\dot{\gamma}|(t)$ exists for $\mathcal{L}^1$-a.e. $t \in (a,b)$, it belongs to $L^1(a,b)$ and it is the least admissible $g$ in \eqref{defac}.
\end{theorem}

\begin{definition}[length]
The length $\ell(\gamma)$ of a curve $\gamma: J \rightarrow E$ is defined as 
$$\ell(\gamma)=\sup \left\{ \sum_{i=0}^{N-1} d(\gamma(t_i), \gamma(t_{i+1})), \, t_0, \ldots, t_N \in J, \, t_0< \cdots < t_N, N \in \N\right\}.$$
\end{definition}

For $\gamma \in \operatorname{AC}_{\mathrm{loc}}(J;E)$, we have the formula $\ell(\gamma)=\int_J |\dot{\gamma}| \de t$.

\begin{definition}[Reparameterization]
If $J,\,\tilde{J}$ are intervals in $\R$, we say that $\tilde{\gamma}: \tilde{J} \rightarrow E$ is a reparametrization of $\gamma: J \rightarrow E$ if there exists an increasing bijection $\phi: \tilde{J} \rightarrow J$ such that $\tilde{\gamma}=\gamma \circ \phi$.
\end{definition}

It will be convenient to localize in the ambient space $E$ the notion of length as follows. When $U \subseteq E$ is an open set, we will make use of the notation $\gamma \, \mres U$ to denote the at most countable family $(\beta_i)_{i \in I}$ of subcurves of $\gamma$ obtained by restricting $\gamma$ to $\gamma^{-1}(U)$, which can be written as a countable disjoint union of open intervals in $J$.

\begin{definition}[Localization in $E$]\label{def:localization_curves}
Given a curve $\gamma:J \rightarrow E$, we define the length of $\gamma$ in a open set $U \subseteq E$ as
$$\ell(\gamma \mres U)=\sum_{i \in I} \ell(\beta_i),$$
where $\gamma \mres U=(\beta_i)_{i \in I}$.
We say that $\gamma$ has locally finite length  in $E$ if $\ell(\gamma \mres U)<\infty$ for any $U\subseteq E$ bounded.
\end{definition}

Note that $\ell(\gamma)$ is invariant under reparametrizations, and hence the same is true for $\ell(\gamma \mres U)$. The localized length can be uniquely extended to a regular Borel measure $F \mapsto \ell(\gamma \mres F)$, $F \subseteq E$ Borel, which in the case $\gamma \in \operatorname{AC}_{\mathrm{loc}}(J;E)$ is given by the formula
$$\ell(\gamma \mres F)=\int_{\gamma^{-1}(F)}|\dot{\gamma}| \de t=\gamma_\#(|\dot{\gamma}| \, \mathcal{L}^1 \mres J)(F) \in [0,\infty].$$

Recall that, for $J\subseteq\R$ Borel,  $\llbracket J \rrbracket\in\mathbf{M}_{1,b}(\R)$ denotes the elementary current associated to
the integration on $J$, namely $\llbracket J \rrbracket(f,\pi)=\int_J f\pi'\de t.$

\begin{definition}\label{defTbeta}
Given $\gamma \in \operatorname{Lip}(J; E)$ with locally finite length in $E$, we associate to it the one-dimensional local current $T_\gamma \in \mathbf{M}_{1,b}(E)$ given by $T_\gamma:=\gamma_\# \llbracket J \rrbracket$, that is
$$T_\gamma(f\de\pi)=\int_J f(\gamma(t)) 
 (\pi\circ\gamma)'(t) \de t, \quad (f, \pi) \in \operatorname{Lip}_{\mathrm{bs}}(E) \times \operatorname{Lip}(E).$$
 \end{definition}

 Assuming without loss of generality $J=(0,1)$, note that if $\gamma$ has finite length then it can be uniquely extended to $[0,1]$ and it holds
 $$\partial T_\gamma=\delta_{\gamma(1)}-\delta_{\gamma(0)}.$$

Using the definition of $\|T_\gamma\|$  one can easily prove that
\begin{equation}
\|T_\gamma\|\leq \gamma_\# \bigl(|\dot\gamma|{\mathcal L}^1\mres J\bigr).
\end{equation}
Moreover, it is worth to mention that the area formula for metric-valued Lipschitz maps (see for instance \cite{AK0}, \cite{AT}) gives
$\gamma_\# \bigl(|\dot\gamma|{\mathcal L}^1\mres J\bigr)= N{\mathcal H}^1\mres\gamma(J)$ with $N(x)={\mathcal H}^0(\gamma^{-1}(x))$.

\section{Cyclic and Acyclic currents}\label{section-acyc}

\begin{definition}\label{defsubcurrent}
   For $T,\,S \in \mathbf{M}_{m,b}(E)$, $S$ is a \textit{subcurrent} of $T$ if
   \begin{equation}\label{eqdefsubcurrent}
   \|T\|=\|S\|+\|T-S\|.
   \end{equation}
   In this case, we denote this relationship by $S\leq T$.
\end{definition}

 \begin{remark}\label{subcurrentsufficopen}
  Note that, to show that \eqref{eqdefsubcurrent} holds, it is enough to check the identity 
  $\|T\|(B)=\|S\|(B)+\|T-S\|(B)$ for all bounded open sets $B$.
  Furthermore, since the inequality $\leq$ always holds, it is enough to check the identity for specific family of 
  bounded open sets exhausting $E$, as the family of open balls with fixed center and diverging radius.
    \end{remark}

\begin{remark}
 For any Borel set $B\subseteq E$, it always holds $T \,\mres B\leq T $.
 \end{remark}
 
 \begin{remark}
        If $R\leq S$ and $S\leq T$, then $R\leq T$. Indeed, the subadditivity of mass gives
        $$ \|T\|=\|S\|+\|T-S\|=\|R\|+\|S-R\|+\|T-S\|\geq \|R\|+\|T-R\|$$
        as well as the validity of the converse inequality.
    \end{remark}

\begin{proposition}\label{sublemma}
    If $R\leq T$ and $S\leq T-R$ then $R\leq R+S$ and $R+S\leq T$.
\end{proposition}
\begin{proof}
    We have
    \begin{align*}
    \|T\|&=\|T-R\|+\|R\|=\\
    &=\|T-R-S\|+\|S\|+\|R\|\geq\|T-R-S\|+\|R+S\|\geq\|T\|.
    \end{align*}
   In particular, all the intermediate inequalities are equalities, and the conclusion follows.
\end{proof}

\begin{proposition}\label{sublemma2}
 Let $T\in\mathbf{M}_{m,b}(E)$, $(S_n)\subseteq \mathbf{M}_{m,b}(E)$, $S_n\leq T$ for all $n\in\mathbb{N}$, and suppose that  $S_n\rightharpoonup S$  weakly as $n \rightarrow \infty$. Then $S \leq T$ and $\|S_n\|(A) \rightarrow 
 \|S\|(A)$ for any bounded open set $A\subseteq E$.
\end{proposition}
\begin{proof}
Consider the sequence $(T-S_n)$, converging weakly to $T-S$. By the lower semicontinuity, for any bounded open set $A$, we have  

\begin{align*}
\|S\|(A)+\|T-S\|(A) & \leq \liminf_{n \rightarrow \infty} \|S_n\|(A)+\liminf _{n \rightarrow \infty} \|T-S_n\|(A)) \\
& \leq \liminf _{n \rightarrow \infty}\left[\|S_n\|(A)+\|T-S_n\|(A)\right]=\|T\|(A)<\infty,
\end{align*}
which means thanks to Remark~\ref{subcurrentsufficopen} that $S \leq T$.
In addition, the two inequalities used in the first line have to be equalities and the first one gives 
$\|S\|(A)=\liminf_n\|S_n\|(A)$. But, since this holds for any subsequence of $(S_n)$, 
we find $\|S\|(A)=\lim_n\|S_n\|(A)$.
\end{proof}

\begin{definition}\label{acyc-def}
$C\in\mathbf{M}_{m,b}(E)$ is called a \textit{cycle} of $T \in \mathbf{M}_{m,b}(E)$ if $C\leq T$ and $\partial C=0$.
We say that $T \in \mathbf{M}_{m,b}(E)$ is \textit{acyclic} if $C=0$ is the only cycle of $T$.    
\end{definition}

\begin{theorem}\label{decompcycacyc}
Every current $T\in \mathbf{M}_{m,b}(E)$ contains a cycle $C$ such that $T-C$ is acyclic.
In particular, any $T\in \mathbf{M}_{m,b}(E)$ can be decomposed as $T=A+C$, where $C$ is a cycle of $T$ and $A$ is acyclic.
\end{theorem}
\begin{proof} We argue as in Section~\ref{Sec_main}, introducing an auxiliary function to evaluate the masses of the subcurrents.
Let $\zeta:E\to [0,1]$ be a continuous function with $\int_E\zeta\de \|T\|<\infty$ and with
$c_n:=\inf_{B_n}\zeta>0$ for all $n\in\mathbb{N}$ where, as usual, $B_n:=B_n(\bar x)$ for some fixed $\bar x\in E$. 
Set
$$
\xi(T):=\sup \left\{\int_E\zeta\de\|C\|\,\mid\,C\leq T,\,\partial C=0\right\}.
$$

Now, if we choose a cycle $C_1\leq T$ such that $\int_E \zeta\de\|C_1\|>\xi(T)/2$, then
$\xi(T-C_1)\leq\xi(T)/2$. Indeed, if the converse inequality were true, we could find $C_2\leq T-C_1$
with $\int_E \zeta\de\|C_2\|>\xi(T)/2$. But then, \cref{sublemma} would give $C_1+C_2\leq T$
and 
$$
\int_E\zeta\de \|C_1+C_2\|=\int_E\zeta\de\|C_1\|+\int_E\zeta\de\|C_2\|>\xi(T),
$$
a contradiction. Proceeding in this way (and assuming that the process never stops, otherwise the
proof is trivial) we can find cycles $C_k$ with $C_0=0$, $\xi(T-\sum_0^kC_i)\leq 2^{-k}\xi(T)$ and 
$$\bigl\Vert\sum_{i=0}^kC_i\bigr\Vert=\sum_{i=0}^k\|C_i\|\leq\|T\|\qquad \forall\,k\in\mathbb{N}.$$
By construction we have
$$
\sum_{i=p}^q\|C_i\|(B_n)\leq
c_n^{-1}\sum_{i=p}^q\int_{B_n}\zeta\de \|C_i\|\leq c_n^{-1}\sum_{i=p}^q\xi(T-\sum_{j=0}^{i-1}C_j)\leq
c_n^{-1}\xi(T)\sum_{i=p}^q 2^{-i}
$$
for all $n,\,p,\,q\in\N$ with $p\leq q$.
Hence, since $n$ is arbitrary, the series $\sum_{i=0}^\infty C_i$ defines a cycle $C\in\mathbf{M}_{m,b}(E)$ with $C\leq T$.
Furthermore, Proposition~\ref{sublemma2} gives that $T-C\leq T-\sum_{i=0}^kC_i$ for all $k\in\N$.

Finally, we conclude the proof by demonstrating that $A=T-C$ is acyclic. We do so by noting that if $\tilde C$ is a cycle of $A$, then $\tilde{C}\leq T-\sum_{i=0}^kC_i$ for all $k$, and since $\xi(T-\sum_{i=0}^kC_i)\to 0$ as $k\to\infty$, we conclude that $\tilde C=0$.
\end{proof}

 In the one-dimensional case $m=1$, we can push forward the decomposition of $T \in \mathbf{N}_{1,b}(E)$ by writing the acyclic part as a sum of subcurrents of finite boundary mass, with superposition holding also at the level of the boundaries. To do this we rely on \cref{PS} in a fundamental way.

\begin{theorem}\label{infinite-boundary}
Let $T\in \mathbf{N}_{1,b}(E)$, then there exist currents
$S_i\in \mathbf{N}_{1,b}(E)$ such that $T=\sum_{i=0}^{\infty}S_{i}$, $\partial S_{0}=0$ and 
\begin{equation}\label{requirements}\|\partial S_{m}\|(E)<\infty,\quad S_{m}\leq T-\sum_{i<m}S_i,\quad 
\partial S_m\leq \partial T-\sum_{i<m}\partial S_{i}\quad \quad \forall\,m \in\mathbb{N}.\end{equation}
\end{theorem}

\FloatBarrier

\begin{proof} We set $S_0=C$, where $C$ is a maximal cycle of $T$ given by \cref{decompcycacyc}.
Now, fix some $\bar{x}\in E$. We will construct inductively the subcurrents $(S_i)_{i \geq 1}$, 
ensuring that they satisfy the requirements (\ref{requirements}) along with the additional condition 
\begin{equation}\label{noboundary}
\|\partial T_{m}\|(B_{m-1}(\bar{x}))=0,
\end{equation}
where $T_{m}:= T-\sum_{i<m}S_i$. Notice that the conditions \eqref{requirements} inductively give 
\begin{equation}\label{eq:inductive}
\sum_{i<m}\|S_i\|\leq\|T\|\qquad\text{and}\qquad\sum_{i<m}\|\partial S_i\|\leq\|\partial T\|,
\end{equation}
so that $\|T_m\|\leq \|T\|$ and $\|\partial T_m\|\leq\|\partial T\|$. \\
Assuming that $S_0,\ldots,S_{m-1}$ have been constructed, we will then find $S_{m}\in\mathbf{N}_{1,b}(E)$ such that 
$$S_{m}\leq T_{m}, \quad \,\partial S_{m}\leq \partial T_{m}, \quad \|\partial S_{m}\|(E)<\infty \quad \textrm{and} \quad \|\partial (T_m-S_m) \|(B_{m}(\bar{x}))=0.$$
Let $r_{m,n}\uparrow\infty$ as $n\to\infty$ be radii such that $T_{m}\mres B_{r_{m,n}}(\bar{x})\in \mathbf{N}_{1}(E)$ for each $n \in \N$ (which exists thanks to \cref{deriv-est}) and that $r_{m,1}\geq m$. \\
For all $n\in\mathbb{N}$ consider the measures $(\eta_{m,n})_n$ obtained by applying \cref{PS} to $T_m\mres{B_{r_{m,n}}}(\bar{x})\in\mathbf{N}_{1}(E)$. That is, $T_m \mres B_{r_{m,n}}(\bar{x})= \int_{\Theta(E)} T_\gamma \de \eta_{m,n}(\gamma)$ with
\begin{align*}
&\|T_m \mres B_{r_{m,n}}(\bar{x})\|=\int_{\Theta(E)} \|T_\gamma\| \de \eta_{m,n}(\gamma), \\
&\| \partial(T_m \mres B_{r_{m,n}}(\bar{x}))\|=\int_{\Theta(E)} \|\partial T_\gamma\| \de \eta_{m,n}(\gamma).
\end{align*}
and $\eta_{m,n}$ is supported in $\Theta(\overline{B}_{r_{m,n}}(\bar{x}))$.
We define the Borel set
$$A_{m}:=\left\{\gamma\in \Theta(E):\, \text{either $\gamma(0)\in B_{m}(\bar{x})$ or $\gamma(1)\in B_{m}(\bar{x})$} \right\}$$
and we consider the currents (see \cref{fig_Smn}.)
$$S_{m,n}:=\int_{A_m} T_\gamma \de \eta_{m,n}(\gamma).$$  \\
\begin{figure}
\includegraphics[width=8cm]{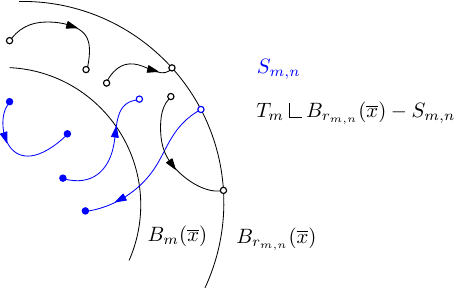}
\captionof{figure}{\bigskip} \label{fig_Smn}
\end{figure}
It is easy to see that
$$
S_{m,n}\leq T_m\mres B_{r_{m,n}}(\bar{x})\leq T_m\quad\text{and}\quad 
\partial S_{m,n}\leq \partial(T_m\mres B_{r_{m,n}}(\bar{x}))$$ 
and also that  $(\partial S_{m,n})\mres B_m(\bar{x})= (\partial T_m)\mres B_{m}(\bar{x})$
for all $n\in\N$ (the equality of boundaries holds because the curves in the complement of $A_m$ 
start and end out of $B_m(\bar{x})$). Also, the local finiteness of $\| \partial T\|$ and the inequality
$\|T_m\|\leq\|T\|$ yield a uniform estimate on the boundary masses:
\begin{equation}\label{uniformbound-boundaries}
\begin{split}
\|\partial S_{m,n} \|(E) &=\int_{A_m} \|\partial T_\gamma\|(E) \de \eta_{m,n}(\gamma)\leq  2\int_{A_m} \|\partial T_\gamma\|(B_m(\bar{x})) \de \eta_{m,n}(\gamma) \\
& \leq 2\|\partial (T_m\mres B_{r_{m,n}}(\bar{x}))\|(B_{m}(\bar{x})) \leq 2\|\partial T\|(B_{m}(\bar{x}))<\infty.
\end{split}
\end{equation}
We observe that, since for all $m,\,n \in \N$ we have
$$\|S_{m,n}\|\leq \|T\| \quad \textrm{and} \quad \|\partial S_{m,n}\|\mres B_{r_{m,n}}(\bar{x})\leq \|\partial T\|\mres B_{r_{m,n}}(\bar{x}),$$
the sequence $(S_{m,n})_n$ satisfies the hypotheses of \cref{comp-theo-loc}, i.e. it is equi-locally bounded in $\mathbf{N}_{1,b}(E)$ and locally equi-tight. Therefore, there exists, up to subsequences, a weak limit that we will denote by $S_{m}$. \\
By \cref{sublemma2}, $S_m\leq T_m$, and by lower semicontinuity of mass and (\ref{uniformbound-boundaries}) we get $\|\partial S_{m}\|(E) \leq 2\|\partial T\|(B_{m}(\bar{x}))<\infty.$ 
Furthermore, fixed any $B$ open bounded, for $n$ sufficiently large we have 
$$(\partial S_{m,n}) \mres B\leq \partial(T_m \mres B_{r_{m,n}}(\bar{x})) \mres B= (\partial T_m) \mres B,$$
and so by applying again \cref{sublemma2} we find that $(\partial S_{m}) \mres B \leq (\partial T_m) \mres B$, hence $\partial S_m \leq \partial T_m$ by \cref{subcurrentsufficopen}. 
We also have
$$\|\partial (T_{m}-S_{m})\|(B_{m}(\bar{x})) \leq \liminf_n \| \partial (T_m-S_{m,n})\|(B_{m}(\bar{x}))=0,$$
that is, $\|\partial T_{m+1}\|(B_{m}(\bar{x}))=0$. \\
Finally, the thesis follows from the fact that the series
$\sum_{i=0}^{\infty}S_{i}$ converges (locally in mass, hence weakly) to a local 1-current $S$, which must coincide with $T$.
Indeed, we have local absolute convergence, since for all $m \in \N$ \eqref{eq:inductive} gives $\sum_{i=0}^{\infty} \|S_i\|\leq \|T\|$. Furthermore, as $T_m \rightharpoonup T-S$, by \cref{sublemma2} we find that $T-S \leq T_1$ and from (\ref{noboundary}) we deduce that $\partial (T-S)=0$. The fact that $T_1=T-S_{0}$ is acyclic then implies that $T=S$.
\end{proof}

 \begin{remark}\label{remark-reduction}
The decomposition $T=\sum_i S_i$ given by \cref{infinite-boundary} satisfies
\begin{equation}\label{double-superpos}
\|T\|=\sum_{i=0}^\infty \|S_i\|, \quad \| \partial T\|=\|\partial(T-S_0)\|=\sum_{i=1}^\infty \| \partial S_i\|.
\end{equation}
In particular, in order to prove \cref{MT}, it is sufficient to show that it is true when $\mathbb{M}(\partial T) < \infty$. Moreover, in the case $\mathbb{M}(T) < \infty$, combining (\ref{double-superpos}) with \cref{PS}, we obtain a measure $\eta$ over $\Theta(\overline{E})$ which decomposes $T$ as in \cref{MT}. Hence, it is necessary to work with the space of open ended curves $\Gamma(E)$ only when $T$ has infinite mass.
 \end{remark}

\section{Proof of the main result}\label{Sec_main}

\subsection{Plan of the proof}

To prove the main result, we will develop a construction which will enable the application of \cref{PS}, which is the already known result for Ambrosio-Kirchheim normal metric currents.
We will accomplish this by turning a given local current $T$ whose boundary has finite mass, which can be assumed with no loss of generality thanks to \cref{infinite-boundary}, into an Ambrosio-Kirchheim normal current, via an ad hoc conformal transformation of the ambient metric.
A key step will be the a priori reduction to the case where the ambient metric space is $\ell^{\infty}$, since this will imply good topological properties of the completion of the space with respect to the new conformal metric.

\subsection{Reduction to the case $\mathbf{E=\ell^\infty}$}\label{riduzioneallinfinito} By our assumption,
see also Remark~\ref{rem-tightness}, the mass of a local current $T\in\mathbf{M}_{1,b}(E)$ is supported on a separable set,
which allows us to reduce the problem to the case of complete separable metric spaces. \\
Now, consider $\ell^\infty=\ell^\infty(\N; \R)$ with the distance $d_\infty$ induced by the usual norm $\|\cdot\|_\infty$. We recall the following classical result (appeared first in \cite{FR}).

\begin{theorem}\label{embedding}
If $(E,d)$ is separable, then there exists an isometric embedding $j: E \rightarrow \ell^\infty$.
\end{theorem}

Hence, taken $j$ as in \cref{embedding}, we observe that it suffices to show the decomposability of $j_{\#}T\in \mathbf{N}_{1, b}(\ell^{\infty})$. Indeed, from our construction (which relies on \cref{PS}) it will follow that almost every curve in the decomposition of $j_{\#}T$ lies in $\operatorname{supp}(T)$, hence, by completeness, in $E$ (here we naturally identify $x \in E$ with $j(x)$). Thus, from now on, we will work with $(E,d)=(\ell^\infty,d_\infty)$.

\subsection{Conformal transformation and properties}\label{conformal-sec}

Let $T \in \mathbf{N}_{1,b}(E)$ be a locally normal 1-current in $E$. We aim to build a new distance $\delta$ on $E$ in order to be able to apply Theorem~\ref{PS} on $T$, at least when $\| \partial T\|(E)<\infty$.
We define $\delta$ by modifying $d$ in a Riemannian-like manner. For all $x,\,y\,\,\in E$, we set the conformal distance\footnote{Here by $|\cdot|$ we denote the norm $\| \cdot \|_\infty$.}
\begin{equation}\label{defdistancedelta}
\delta(x,y):=    \inf \left\{ \int_0^1 g(|\gamma(t)|)|\dot\gamma|(t) \,\de t\mid\,\,\gamma\,\,\text{$d$-Lipschitz\,},\gamma(0)=x,\,\gamma(1)=y\right\},
\end{equation}
where the  continuous function $g:[0. \infty) \to [0,\infty)$ is built as follows. Choose $r_n \rightarrow \infty$ such that the function $\alpha(r):=\|T\|(\overline{B}_r)$ is differentiable at $r=r_n$ and $\| \partial T\|(\partial B_{r_n})=0$, so that (see \cref{deriv-est})
$T \mres \overline{B}_{r_n}\in\mathbf{N}_1(E)$ and
\begin{equation}\label{r_n}
\| \partial(T \mres \overline{B}_{r_n})\|(\partial B_{r_n}) \leq \alpha'(r_n) < \infty \quad \mbox{for all }\, n.
\end{equation}
We define
\begin{equation}
    \phi(r):=\max\left\{1,\,\|T\|(\overline{B}_{r+1}), \, \max_{r_n \leq r} \alpha'(r_n)  \right\}
\end{equation}
with
\begin{equation}\label{defg} 
g(r)\,:=\frac{1}{\tilde{\phi}(r)2^{r}}.
\end{equation}
where $\tilde{\phi}\geq\phi$ is a continuous function, which we can choose non-decreasing (as $\phi$ is non-decreasing).\\

We will denote by $\overline{E}$ the completion of $E$ with respect to the distance $\delta$ and let
    $$i:(E,d)\to (\overline{E},\delta)$$
denote the inclusion map. In addition, we will use the following notation:
    \begin{itemize}
         \item $I=\overline{E}\setminus i(E)$ denotes the set of points ``at infinity'' resulting from the completion;
        \item $B_{r}(x)$ and ${B}_{r}^\delta(x)$ will be the open balls of radius $r$ and center $x$ in $(\ell^\infty,d)$ and $(\overline{E},\delta)$ respectively; when not specified, it is understood that $x=0$;
        \item $\operatorname{Lip}$ and $\operatorname{Lip}^\delta$ will denote the Lipschitz functions in $(E,d)$ and $(\overline{E},\delta)$ respectively.
        \item $\mathbb{M}$ and $\mathbb{M}_\delta$, as well as $\| \cdot \|$ and $\| \cdot \|_\delta$, 
        will denote the total mass and the mass measures of currents in $(E,d)$ and $(\overline{E},\delta)$ respectively.
        \item Similarly, $\ell(\gamma)$ and $\ell_\delta(\gamma)$ will denote the length of a curve $\gamma$ (and $|\dot{\gamma}|$, $|\dot{\gamma}|_\delta$ its metric derivative) in $(E,d)$ and $(\overline{E},\delta)$ respectively.
    \end{itemize} 

We will usually identify $x$ with $i(x)$ and $E$ with $i(E)$.

\begin{proposition}[Comparison of distances in $E$]\label{Stimeddelta}
The two distances satisfy the following local inequalities:
\begin{itemize}
    \item[(1)] for all $r>0$ there exists $c(r)<\infty$ such that  $$c(r)\delta(x,y)\geq d(x,y)\geq \delta(x,y)\qquad\,\forall\,x,\,y\in B_r;$$
    \item[(2)] for all $z \,\in\, E$, $\eps\in (0,|z|/2)$ one has 
    \begin{equation}\label{inequalityddelta3}
        g(|z|-\eps)d(x,y)\geq \delta(x,y)\qquad\forall x,\,y\,\in\,B_{\eps}(z);
    \end{equation}
    \item[(3)] for all $z \,\in\, E$ and $\eps>0$ small enough, there exists $\overline{\eps}=\overline{\eps}(|z|,\eps)>0$ such that
    \begin{equation}\label{inequalityddelta4}
        \delta(x,y)\geq g(|z|+\overline{\eps})d(x,y)\qquad\forall x,\,y\,\in\,B_{\eps}(z).
    \end{equation}
    Furthermore, $\operatorname{lim}_{\eps\to 0}\overline{\eps}(r,\eps)=0$ for any $r>0$.
\end{itemize}
\end{proposition}
\begin{proof}
The inequality $d\geq\delta$ easily follows from the fact that $g\leq1$. To show the converse inequality in (1), we consider an arbitrary $\gamma$ such that $\gamma(0)=x$, $\gamma(1)=y$. If $\operatorname{supp}(\gamma)\subseteq B_{r+1}$, then we have 
    $$\int_{0}^{1}g(|\gamma(t)|) \,|\dot\gamma|(t) \,\de t \geq g(r+1)\,\ell(\gamma)\geq g(r+1)d(x,y).$$
    On the other hand, if there exists $\overline{t}>0$ such that $\gamma(\overline{t})\notin B_{r+1}$, then 
    $$\int_{0}^{1}g(|\gamma(t)|) \,|\dot\gamma|(t) \,\de t\geq \int_{0}^{\overline{t}}g(|\gamma(t)|) \,|\dot\gamma|(t) \,\de t\,\geq $$
    $$\geq g(r+1) \, d(x,\gamma(\overline{t}))\geq g(r+1)\geq \frac{g(r+1)}{2r}d(x,y).$$
    The inequality then follows by choosing $c(r):=\max\{1,2r\}/g(r+1)$. \\
    For the inequality \eqref{inequalityddelta3}: we have
    $$\delta(x,y)\leq \int_{0}^{1}g(|ty+(1-t)x|)|y-x|\,\de t \leq g(|z|-\eps) d(x,y)$$
    where we used that $ty+(1-t)x \in B_{\eps}(z)$ for all $t \in [0,1]$ and that $g$ is decreasing.\\
    For the last inequality we set $r=|z|$ and
    \begin{equation}\label{inf_eps}
    \overline{\eps}(r,\eps):=\inf\bigl\{h>\eps: \, \int_{r+\eps}^{r+h}g(s) \de s\geq 4\eps\bigr\}.
    \end{equation}
    Clearly $\overline{\eps}(r,\eps)$ is well defined for $\eps$ small enough and goes to 0 as $\eps\to 0$.
    Note that, similarly to what we did for one of the inequalities in (1), for any $\gamma$ connecting $x$ and $y$ such that $\operatorname{supp}(\gamma)\cap B^{c}_{r+\overline
    {\eps}(r,\eps)}\neq\varnothing$, we have 
    $$ \int_{0}^{1}g(|\gamma(t)|)|\dot{\gamma}|(t) \de t\,\geq \int_{r+\eps}^{r+h}g(s)\, \de s\geq 4\eps\geq 2d(x,y)\geq 2\delta(x,y),$$
    for some $h>\eps$ as in the definition of the infimum (\ref{inf_eps}). Here the first inequality follows by the co-area formula applied to the Lipschitz function $f(t):=|\gamma(t)|$.\\
    Hence, in the calculation of $\delta(x,y)$ we may restrict our attention to curves that are supported within 
    $B_{r+\overline{\eps}(r,\eps)}$. In particular, we obtain
   \begin{equation*}
   \begin{split}
   \delta(x,y) &\geq g(r+\overline{\eps}(r,\eps)) \inf \left\{ \int_{0 }^{1}|\dot\gamma|(t) \,\de t: \,\gamma(0)=x,\,\gamma(1)=y \right\} \\
   &=g(r+\overline{\eps}(r,\eps))d(x,y). \qedhere
    \end{split}
    \end{equation*} 
\end{proof}

\begin{remark}\label{remark-eps}
The function $\overline{\eps}(r,\eps)$ we found in \cref{Stimeddelta} (defined by \eqref{inf_eps}) is non-decreasing in $r$.
\end{remark}

\begin{remark}
    From the previous estimates, it follows that the topology of $E$ coincides with the topology induced by
    ($\overline{E},\delta)$ on $E$. In particular, $i(\overline{B}_r)$ is complete in $(\overline{E},\delta)$, 
    $i(B_r)$ is an open set in $\overline{E}$ and $i(B_r)\cap I=\varnothing$ for all $r>0$. 
    \end{remark}

As an immediate consequence of \cref{Stimeddelta}, we obtain that the length of a curve in $(E, \delta)$ is actually given by the formula suggested by (\ref{defdistancedelta}).

\begin{proposition}\label{length-formula}
Let $\gamma \in \operatorname{AC}_{\mathrm{loc}}(J; (E, d))$ (thus also $\gamma \in \operatorname{AC}_{\mathrm{loc}}(J; (E, \delta))$), where $J\subseteq \R$ is an interval. Then
\begin{equation}
 \ell_\delta(\gamma)=\int_J g(|\gamma(t)|) \,|\dot{\gamma}|(t) \de t. 
\end{equation}
\end{proposition}

\begin{proof}
Since $\ell_\delta(\gamma)=\int_J |\dot{\gamma}|_\delta(t) \de t$, it suffices to show that
\begin{equation}\label{metric-der}
|\dot{\gamma}|_\delta(t)=g(|\gamma(t)|) \, |\dot{\gamma}|(t), \quad \mbox{for a.e. } t \in J.
\end{equation}
An application of \eqref{inequalityddelta3} of \cref{Stimeddelta} with $z=x=\gamma(t)$, $y=\gamma(t+h)$ yields
\begin{equation}\label{length-upp}
\delta(\gamma(t), \gamma(t+h)) \leq g(r-\eps)d(\gamma(t), \gamma(t+h)),
\end{equation}
where $r=|\gamma(t)|$ and $\eps=\eps(h)=d(\gamma(t), \gamma(t+h)) \rightarrow 0$ as $h \rightarrow 0$. 
On the other hand, applying \eqref{inequalityddelta4}, again with $z=x=\gamma(t)$, $y=\gamma(t+h)$, we obtain
\begin{equation}\label{length-low}
\delta(\gamma(t), \gamma(t+h))  \geq g(r+\overline{\eps}(r, \eps)) \, d(\gamma(t), \gamma(t+h)).
\end{equation}
Hence, passing to the limit $h \rightarrow 0$ in (\ref{length-upp}) and (\ref{length-low}) and 
using the continuity of $g$ we get (\ref{metric-der}) at any $t\in J$ where the metric derivatives $|\dot{\gamma}|(t)$, $|\dot{\gamma}|_\delta(t)$ exist.
\end{proof}

One of the key advantages of working with $E=\ell^{\infty}$ is the validity of the following lemma.
\begin{lemma}\label{single-point}
$I=\overline{E} \setminus i(E)$ consists of exactly one point.
\end{lemma}
\begin{proof}
We first  consider  the sequence $(c_n)\subseteq E=\ell^{\infty}$, where $c_n=(n,0,0,\ldots)$.
The sequence is Cauchy in $(E,\delta)$ for any $g$ as in \eqref{defg}, since (for $m>n$)
$$\delta(c_{n},c_{m})\leq \int_n^m g(t)\,\de t \leq \int_n^\infty 2^{-t} \de t,$$
therefore it defines a point $x_\infty\in I$.
We are now going to show that any sequence $(a_n)$ converging to an element 
$\tilde{x}_\infty\in I$ gets arbitrarily close to the sequence $(c_n)$, this will provide the equality
$x_\infty=\tilde{x}_\infty$.
Let $(a_n)$ be such a sequence, with $a_n=(a_{n}^{1},a_{n}^{2},\dots)$. Since by \cref{Stimeddelta} the topologies coincide on $d$-balls, we may assume $\operatorname{lim}_{n\to\infty}\|a_n\|_\infty=\infty$.\\
It will then be enough to show that for any $a=(a^{n})_n$ such that $N<\|a\|_{\infty}<2N$ it holds $\delta(a, c_N) \leq \eps(N)$, for some $\eps(N) \rightarrow 0$ as $N\rightarrow \infty$.
Clearly we can find  $K\in\mathbb{N}$ such that $N<\|a^{K}\|<2N$.  
We will estimate $\delta(a,c_{N})$ through the concatenation of the curves $\gamma_{i}=(\gamma_{i}^{j}):[0,1]\to \ell^{\infty}$ $i=1,2,3$ which connects the two points.\\
We set
\begin{align*}
&\gamma_{1}^{m}(t) =
\begin{cases}
(1-t)a^{m} & \text{if } m \neq K, K+1 \\
a^{K} & \text{if } m = K \\
tN + (1-t)a^{K+1} & \text{if } m = K+1
\end{cases}
\\
&\gamma_{2}^{m}(t) =
\begin{cases}
(1-t)\gamma_{1}^{1}(1)+tN & \text{if } m =1 \\
\gamma_{1}^{m}(1) & \text{else } \\
\end{cases}
\\
&\gamma_{3}^{m}(t) =
\begin{cases}
N & \text{if } m =1 \\
(1-t)\gamma_{1}^{m}(1) & \text{else. } \\
\end{cases}
\end{align*}
We notice that $\gamma_{1}(0)=a$, $\gamma_{1}(1)=\gamma_{2}(0)$, $\gamma_{2}(1)=\gamma_{3}(0)$ and $\gamma_{3}(1)=c_{N}$, furthermore $\|\gamma_{i}(t)\|_{\infty} \geq N$ and $|\dot{\gamma_{i}}|(t)\leq 3N$ for all $t,i$.\\
We see that $\gamma_1 \oplus \gamma_{2} \oplus \gamma_{3}$ is a Lipschitz curve connecting $a$ and $c_{N}$, therefore we can conclude the proof with the following estimate:
$$
\delta(a, c_N)\leq\sum_{i=1}^3\int_0^1 g(\|\gamma_i(t)\|_\infty) |\dot{\gamma}_i(t)| \, \de t 
\leq 9N \, g(N) \to 0 \quad \text{as} \quad N \to \infty.\qedhere
$$
\end{proof}

 The unique ``point at infinity'' added by completing the space $(E, \delta)$ will be denoted by $x_\infty$.

\begin{lemma}\label{boundtoinfinity}
For all $r>0$, $x\in\partial B_r$ there exists a curve $\gamma_{x}:[0,1]\to \overline{E}$ such that $\gamma_{x}(0)=x$, $\gamma_{x}(1)=x_{\infty}$, $\operatorname{supp}\gamma_{x}\subseteq B_{r}^{c}$ and 
$\|T_{\gamma_{x}}\|_{\delta}(\overline{E})=\ell_{\delta}(\gamma_{x})=\int_r^\infty g(s)\,\de s$.
\end{lemma}

\begin{proof}
The statement follows with the arc
$$\tilde\gamma(t) =
\begin{cases}
\frac{1}{1-t}x & \text{if } t \,\in\,[0,1), \\
x_{\infty} & \text{if } t =1.
\end{cases}$$
It is easy to see that $\tilde\gamma$ is continuous in $[0,1]$ with respect to
$(\overline{E},\delta)$ and locally absolutely continuous in $[0,1)$ with respect to $d$. Then, using \cref{length-formula}, one finds that 
$$
\ell_{\delta}(\tilde\gamma_x)=\int_0^1 g\left(\frac{|x|}{1-t}\right)\frac{|x|}{(1-t)^2}\de t= \int_r^\infty g(s)\,
\de s.\qedhere$$

\end{proof}
\begin{remark}
 \cref{single-point} and \cref{boundtoinfinity} were the only instances where we explicitly relied on the specific topological structure of $\ell^{\infty}$. The rest of the construction we performed above would still be meaningful in any geodesic metric space $(E,d)$ (which somewhow justifies the use of this generic notation for $(\ell^\infty, d_\infty)$).

\end{remark}
We now examine how the mass measure of local 1-currents is affected by the change of metric, exploiting \cref{Stimeddelta}.

\begin{proposition}\label{estimatefinitemass}
    Let $S\in\mathbf{M}_{1,b}(E)$, then 
 $$\|S\|_\delta=g(|x|) \|S\|.$$
\end{proposition}

\begin{proof}
It is enough to prove the following estimate: for any $0<R_{1}<R_{2}$,
\begin{equation}\label{stimacorone}
 g(R_{2}) \, \|S\|(\overline{B}_{R_{2}}\setminus \overline{B}_{R_{1}})\leq \|S\|_\delta(\overline{B}_{R_2} \setminus \overline{B}_{R_1}) \leq g(R_{1}) \, \|S\|(\overline{B}_{R_{2}}\setminus \overline{B}_{R_{1}}).
 \end{equation}

 Indeed, assume that (\ref{stimacorone}) holds. Applying it to $S \mres B$, where $B$ is an arbitrary bounded Borel set in $E$, let us say $B \subseteq \overline{B}_R$, we obtain
 \begin{equation*}
 \begin{split}
  \|S\|_\delta(B)&=\sum_{i=0}^n \|S\|_\delta(B \cap \overline{B}_{R \frac in} \setminus \overline{B}_{R \frac {i-1}n}) \geq \sum_{i=0}^n g\left(R \frac in\right)\|S\|(B \cap \overline{B}_{R \frac in} \setminus \overline{B}_{R \frac {i-1}n}) \\
  & = \int g\left(\frac Rn \Bigl\lceil \frac {|x|}{R/n} \Bigr\rceil \right) \chi_B(x) \de \|S\|(x),
 \end{split}
 \end{equation*}
 for all $n \in \N$, which passing to the limit $n \rightarrow \infty$ (by dominated convergence) yields
 $$\|S\|_\delta(B) \geq \int_B g(|x|) \de\|S\|(x),$$
 proving that $\|S\|_\delta \geq g(|x|)\|S\|$. The converse inequality can be shown in the same way, using the other estimate in (\ref{stimacorone}). \\
 Let us now prove (\ref{stimacorone}) using directly the definition of mass.
 Fix any $0<\eps\leq R_1/2$. By the tightness assumption on $\|S\|$, for any $\eps>0$, we can find a compact set $K\subseteq B_{R_2+\eps}\setminus \overline{B}_{R_{1}+\eps}$ such that 
    $$\|S\|_{\delta}(B_{R_2+\eps}\setminus (\overline{B}_{R_{1}+\eps}\cup K))\leq \eps .$$
Consider now a cover of $K$ made by open balls $B_{\eps}(x_n)$, $1\leq n\leq N$, with center $x_n\in K$.
Let $\{\phi_n\}$, $0\leq n\leq N$, be a Lipschitz partition of unity in $E$ associated to the covering 
$\overline{E}\setminus\bigcup_{n=1}^{N}\overline{B}_{\eps}(x_{n})$, $B_{2\eps}(x_1),\ldots, B_{2\eps}(x_N)$ and note that $\sum_{n=1}^{N}\phi_n \equiv 1$ on the union of the balls $B_{\eps}(x_n)$. Recalling Remark~\ref{rem_dualmass}, to estimate from above the quantity
$$\|S\|_\delta\left(\bigcup_{n=1}^{N} B_{\eps}(x_n)\right),$$
we consider an arbitary family $(f_\lambda, \pi_\lambda)_{\lambda \in \Lambda} \subseteq \operatorname{Lip}^\delta_{\mathrm{bs}}(\overline{E}) \times \operatorname{Lip}^\delta_1(\overline{E})$, for some finite set $\Lambda$, such that $\operatorname{supp}(f_\lambda) \subseteq \bigcup_{n=1}^{N} B_{\eps}(x_n)$ for all $\lambda \in \Lambda$ and $\sum_{\lambda \in \Lambda} |f_\Lambda| \leq 1$. We have
\begin{equation*}
\begin{split}
\sum_{\lambda \in \Lambda} S(f_\lambda, \pi_\lambda)&=
\sum_{\lambda \in \Lambda} \sum_{n=1}^{N} S(f_\lambda \phi_n, \pi_\lambda)
=g(R_1 - \eps)  \sum_{n=1}^{N} \sum_{\lambda \in \Lambda} S\left(f_\lambda \phi_n, \tilde{\pi}_\lambda\right) \\
&\leq g(R_1 - \eps) \, \sum_{n=1}^N\|S\| \left( B_{\eps}(x_n) \right)=
g(R_1 - \eps) \, \|S\| \left( \bigcup_{n=1}^{N} B_{\eps}(x_n) \right),
\end{split}
\end{equation*}
 where we used the locality property and the functions 
$\tilde\pi_\lambda\in \operatorname{Lip}_1(\overline{E})$, $1$-Lipschitz extensions of the
restriction of $\pi_\lambda/g(R_1-\eps)$ to $B_{2\eps}(x_n)$. The latter function is $1$-Lipschitz thanks to
\eqref{inequalityddelta3} of \cref{Stimeddelta}. Hence,
$$\|S\|_\delta\left(\bigcup_{n=1}^{N} B_{\eps}(x_n)\right) \leq g(R_1 - \eps) \, \|S\| \left( \bigcup_{n=1}^{N} B_{\eps}(x_n) \right).$$
Finally we have 
    \begin{align} \label{S_delta}
     \|S\|_\delta(\overline{B}_{R_2} \setminus \overline{B}_{R_1+\eps})
     &\leq
     \|S\|_{\delta}(B_{R_2+\eps}\setminus \overline{B}_{R_{1}+\eps})
     \leq
     \eps+\|S\|_{\delta}(K)\nonumber \\
     &\leq
     \eps+\|S\|_{\delta}\left(\bigcup_{n=1}^{N}B_{\eps}(x_{n})\right) \leq
      \eps+ g(R_{1}-\eps) \, \|S\|\left(\bigcup_{n=1}^{N}B_{\eps}(x_{n})\right) \\
      &\leq
     \eps+ g(R_{1}-\eps)  \, \|S\|(B_{R_2+2\eps}\setminus \overline{B}_{R_{1}})\nonumber .
     \end{align}
     Then, by sending $\eps\to 0$, we infer the inequality
     $$\|S\|_\delta(\overline{B}_{R_2} \setminus \overline{B}_{R_1}) \leq g(R_{1}) \, \|S\|(\overline{B}_{R_{2}}\setminus \overline{B}_{R_{1}}).$$
     The converse inequality can be proved analogously with \eqref{inequalityddelta4} of \cref{Stimeddelta}:
     for $\eps\in (0,1)$ small enough to apply \cref{inequalityddelta4} at $r\leq R_{2}+1$, we find  a compact set $K\subseteq B_{R_2 +\eps}\setminus \overline{B}_{R_{1}+\eps}$ such that 
    $$\|S\|(B_{R_{2}+\eps}\setminus (\overline{B}_{R_{1}+\eps}\cup K))\leq \eps .$$
    Considering a finite covering of $K$ composed of open balls $B_{\eps}(x_n)$, $1\leq n\leq N$, each of radius $\eps$ and center $x_n\in K$ and using a suitable partition of unity we find, analogously to the already discussed case (here we also use \cref{remark-eps}),
$$
\|S\| \left( \bigcup_{n=1}^{N} B_{\eps}(x_n) \right)\leq \frac{1}{g(R_2 + \eps+\overline{\eps}(R_{2},2\eps))} \|S\|_{\delta} \left( \bigcup_{n=1}^{N} B_{\eps}(x_n) \right)
$$
and then one concludes by performing a computation analogous to (\ref{S_delta}) and sending $\eps\to0$.\qedhere
\end{proof}

\begin{remark}\label{0-mass unchanged}
Note that the 0-currents in $E$ can be canonically identified the 0-currents in $(\overline{E}, \delta)$ supported in $i(E)$, and their masses are not affected by the change of metric.
\end{remark}

The following result, which shows that the choice of the ``optimal'' curves is somehow independent of the metric, will be an important ingredient in the conclusion of the proof of \cref{MT}.

\begin{lemma}\label{no-canc}
Let $S \in \mathbf{M}_1(\overline{B}_R, d)$ (and so also $S\in \mathbf{M}_1(\overline{B}_R, \delta)$) be such that there exists a positive Borel measure $\eta$ in $\Theta(\overline{B}_R)$ such that
$$
S=\int_{\Theta(\overline{B}_R)} T_\gamma \de\eta(\gamma).
$$
Then
$$
\|S\|=\int_{\Theta(\overline{B}_R)} \| T_\gamma \| \de\eta(\gamma)
\quad\text{if and only if}\quad
 \|S\|_\delta=\int_{\Theta(\overline{B}_R)} \| T_\gamma \|_\delta \de\eta(\gamma).
$$
\end{lemma}

\begin{proof}
Assume that $\|T\|=\int_{\Theta(\overline{B}_R)} \| T_\gamma \| \de\eta(\gamma)$. Then, applying \cref{estimatefinitemass} to the 1-currents $S$ and $T_\gamma$, $\gamma \in \Theta(\overline{B}_R)$, we get
$$\|S\|_\delta=g(|x|) \|S\|=\int_{\Theta(\overline{B}_R)} g(|x|) \, \|T_\gamma\| \de \eta(\gamma)=\int_{\Theta(\overline{B}_R)}\|T_\gamma\|_\delta \de \eta(\gamma).$$
The converse implication is proven in the same way, using that $\|S\|=(g(|x|))^{-1} \|S\|_\delta$ (as $g>0$).
\end{proof}
\subsection{Extension of $\mathbf{T}$ to $\mathbf{(\overline{E}, \delta)}$}
In this section we show how the conformal change of metric described in \cref{conformal-sec} allows us to extend $T$ to a normal current on $(\overline{E}, \delta)$, to which we will be able to apply \cref{PS}. We shall define
\begin{equation}\label{extension of T} \overline{T}:=\lim_{r\to \infty} i_{\#}(T\mres \overline{B}_r),\end{equation}
where the limit is in mass. In particular, it will follow that $\overline{T}$ has finite mass.

\begin{proposition}\label{finite-deltamass}
If $T\in\mathbf{N}_{1,b}(E) $ the limit  $\lim_{r\to \infty} i_{\#}(T \,\mres \overline{B}_r)$ is a well-defined element
$\overline{T}$ of $\mathbf{M}_1(\overline{E})$ (in particular $\overline{T}$ has finite mass).
\end{proposition}

\begin{proof}
Applying \cref{estimatefinitemass} (actually just estimate (\ref{stimacorone})) with $S=T$ and using that
$g(r) \leq 2^{-r} (\|T\|(\overline{B}_{r+1}))^{-1}$ (recall its definition in \eqref{defg}) for $r \in \N$ one infers that
$$
\|T\|_\delta(\overline{B}_{r_2} \setminus\overline{B}_{r_1}) \leq \sum_{r= \lfloor r_1 \rfloor}^\infty \|T\|_\delta(\overline{B}_{r+1} \setminus \overline{B}_r) \leq \sum_{r= \lfloor r_1 \rfloor}^\infty 2^{-r}\qquad\forall r_2\geq r_1.
$$
Hence, for any sequence $(r_n)$ with $r_n \rightarrow \infty$ the sequence $(i_\#(T \mres \overline{B}_{r_n}))$ is Cauchy in mass and the limits are all the same.
\end{proof}

Now, it remains to verify that $\partial \overline{T}$ still has finite mass, if $\partial T$ does so. Note that, for this purpose, the topological properties of $E=\ell^{\infty}$ described in \cref{single-point} and \cref{boundtoinfinity} are essential for the validity of the result, as shown by the following example.

\begin{example}\label{exnofinitemass}
Let $E:=([0,\infty)\times\{0\})\cup\bigcup_{n\in\mathbb{N}}(\{n\}\times [0,\infty))\subseteq\mathbb{R}^{2}$ with the induced Euclidean distance and consider the geodesic distance on $E$ defined by the relation $$d(x,y):=\inf \left\{ \ell(\gamma)\,\,\,\mid \gamma\in \operatorname{Lip}([0,1],E), \,\gamma(0)=x,\,\gamma(1)=y\right\}.$$
Consider the current $T:=\sum_{n=0}^{\infty}\llbracket T_{\gamma_n}\rrbracket$, with $\gamma_n$ as in \cref{fig_gamman}. It is easy to see that $T$ has locally finite mass and that $\partial T=0$. On the other hand, the
corresponding current $\overline{T}$ in $(\overline{E},\delta)$ constructed as above does not have finite mass.

\smallskip 
\begin{center}
\includegraphics[width=5cm]{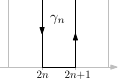}
\captionof{figure}{\bigskip}\label{fig_gamman}
\end{center}
\end{example}

As anticipated, under our assumptions the current $\partial\overline{T}$ has finite mass. More precisely, we have the following statement.

\begin{proposition} \label{boundary-est}
For all $T \in \mathbf{N}_{1, b}(E)$ one has $\mathbb{M}_\delta(\partial \overline{T}) \leq 2 \mathbb{M}(\partial T)$.
\end{proposition}
\FloatBarrier

\begin{proof}
Let $T_r:=T \, \mres \overline{B}_r$. We work with $r \in \{r_n\}$ (without explicitly relabeling the sub-sequence), so that (\ref{r_n}) holds. In particular, $T_r$ is normal, hence by \cref{PS} we can write $T_r=A_r+C_r$, where $\partial C_r=0$, $C_r \leq T_r$ and $A_r$ acyclic. Also, the current $A_r$ enjoys the representation $\int_{\Theta(\overline{B}_r)} T_\gamma\de\eta_r(\gamma)$ with
\begin{equation}\label{superpos_r}
\|A_r\|=\int_{\Theta(\overline{B}_r)} \|T_\gamma\| \de\eta_r(\gamma), \qquad
 \|\partial A_r\|=\|\partial T_r\|=\int_{\Theta(\overline{B}_r)} \|\partial T_\gamma\| \de\eta_r(\gamma).
\end{equation}
Set
$$E_r:=\bigl\{ \gamma \in \Theta(\overline{B}_r) \mid \gamma(0)\in\partial B_r,\,\, \gamma(1) \in \partial B_r\bigr\}.$$
Now, we construct a new current $\widetilde{T}_r$ by modifying those curves in the decomposition above which lie in $E_r$. Namely, given a curve $\gamma \in E_r$ we replace it with the closed curve  $\alpha(\gamma)=\alpha_1(\gamma) \oplus \gamma \oplus \alpha_2(\gamma)$, where $\alpha_1(\gamma)$, $\alpha_2(\gamma)$ denote the two half-lines from the origin which connect $x_\infty$ to $\gamma(0)$ and $\gamma(1)$ to $x_\infty$ respectively, see \cref{fig_alpha}. \\
\begin{figure}
\centering
\begin{minipage}{.5\textwidth}
  \centering
  \includegraphics[width=5cm]{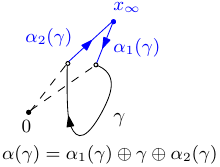}
  \captionof{figure}{}
  \label{fig_alpha}
\end{minipage}%
\begin{minipage}{.5\textwidth}
  \centering
  \includegraphics[width=6cm]{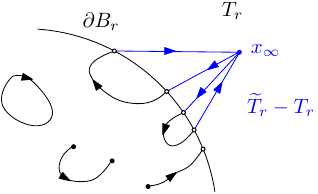}
  \captionof{figure}{}
  \label{fig_Ttilde}
\end{minipage}
\end{figure}
That is, we define (see also \cref{fig_Ttilde})
$$\widetilde{T}_r:=\int_{E_r} T_{\alpha(\gamma)} \de \eta_r(\gamma)+\int_{E_r^c} T_\gamma \de\eta_r(\gamma)+C_r,$$

so that
$$
\partial \widetilde{T}_r:=\int_{E_r^c} \partial T_\gamma \de\eta_r(\gamma).
$$
Note that here we heavily exploited the topological information given by \cref{single-point}. Also, the linear structure of $E=\ell^\infty$ allowed us to use curves $\alpha_1$, $\alpha_2$ which satisfy \cref{boundtoinfinity} (which will be useful in the following), namely half-lines from the origin. \\
We claim that
\begin{equation}\label{mass-convergence}
\lim_{r \rightarrow \infty} \mathbb{M}_\delta(T_r-\widetilde{T}_r)=0.
\end{equation}
Assuming that (\ref{mass-convergence}) is true, then the mass of $\partial \overline{T}$ can be controlled using only the boundary mass of those curves in $E_r^c$ which, by definition, have at least an extremal in the interior of
$\overline{B}_r$, and thus contribute to $\|\partial T\|(B_r)$. Indeed,

\begin{eqnarray*}
\|\partial  \widetilde{T}_r\|(\overline{E})&\leq&
\int_{E_r^c} \|\partial T_\gamma\|(\overline{E}) \de\eta_r(\gamma)\leq
2\int_{E_r^c} \|\partial T_\gamma\|(B_r) \de\eta_r(\gamma)\\
&\leq&
2\int_{\Theta(\overline{B}_r)} \|\partial T_\gamma\|(B_r) \de\eta_r(\gamma)=
2\|\partial T_r\|(B_r)\leq 2\mathbb{M}(\partial T)
\end{eqnarray*}
since $(\partial T_r)\mres B_r=(\partial T)\mres B_r$.
To be more precise, since $T_r \rightarrow \overline{T}$ in mass in $\mathbf{M}_1(\overline{E})$ by construction, then (\ref{mass-convergence}) immediately implies
$\mathbb{M}_\delta(\widetilde{T}_r-\overline{T}) \rightarrow 0$, which gives us, for any $f \in \operatorname{Lip}^\delta(\overline{E})$ with $\sup |f |\leq 1$,
$$
\partial \overline{T}(f)  = \overline{T}(1\,\de f)=
\lim_r\widetilde{T}_r(1\, \de f) = \lim_r\, \partial\widetilde{T}_r(1\,\de f) \leq 2\mathbb{M}(\partial T).$$
The convergence (\ref{mass-convergence}) can be proved through a direct estimate, using the fact that (by construction) $g(t) \leq 2^{-t}(\frac\de{\de r} \|T\|(\overline{B}_r))^{-1}$ for any $r \in \{r_n\}$, $t \geq r$. Indeed
\begin{equation*}
\begin{split}
\mathbb{M}_\delta(T_r-\widetilde{T}_r)& \leq \int_{E_r} (\|T_{\alpha_1(\gamma)}\|_\delta+\|T_{\alpha_2(\gamma)}\|_\delta) \de \eta_r(\gamma)=2\eta_r(E_r) \,\int_r^\infty g(t) \de t \\
& \leq \| \partial T_r\|(\partial B_r) \int_r^\infty g(t) \de t  \leq \frac\de{\de r} \|T\|(\overline{B}_r) \int_r^\infty g(t) \de t \\
&\leq \int_r^\infty 2^{-t} \de t \rightarrow 0\quad\text{as $r\to\infty$},
\end{split}
\end{equation*}
where we also exploited the following consequence of \eqref{superpos_r}:
\begin{equation*}
\|\partial T_r\|(\partial B_r) \geq \int_{E_r} \| \partial T_\gamma\|(\partial B_r) \de \eta_r(\gamma)=2\eta_r(E_r). \qedhere
\end{equation*}
\end{proof}

\begin{corollary}\label{cor}
For any $T \in\mathbf{N}_{1,b}(E)$ such that $\| \partial T\|(E) < \infty$, the current $\overline{T}$ built by $\lim_r T \, \mres B_r$ belongs to $\mathbf{N}_1(\overline{E})$. In particular, there exists a measure $\overline{\eta}$ over $\Theta(\overline{E})$ that decomposes $\overline{T}$ into curves as in \cref{PS}.
\end{corollary}

We also notice that, as a consequence of \cref{no-canc}, acyclicity is preserved by the transformation $T \mapsto \overline{T}$.

\begin{lemma}\label{acyc-lemma}
If $T \in \mathbf{N}_{1,b}(E)$ is acyclic, then $\overline{T}$ is acyclic in $\overline{E}$.
\end{lemma}

\begin{proof}
Let $S \in \mathbf{M}_1(\overline{E})$ s.t. $S \leq \overline{T}$ and $\partial S=0$. Then $S$ naturally defines a local current on $E$ with no boundary (since any $f \in \operatorname{Lip}(E)$ with bounded support is also an element of $\operatorname{Lip}^\delta(\overline{E})$). Recalling Remark~\ref{subcurrentsufficopen}, to obtain that $S=0$ from the acyclicity of $T$ it remains to check that
$$\|S\|(\overline{B}_R)+\|T-S\|(\overline{B}_R)=\|T\|(\overline{B}_R)$$
for arbitrarily large $R$. By hypothesis, we have the identity of measures
\begin{equation}\label{S,T delta}
\|S\|_\delta+\|\overline{T}-S\|_\delta=\|\overline{T}\|_\delta.
\end{equation}
Apply then \cref{PS} to find, for a.e. $R>0$, two measures $\eta_1$, $\eta_2$ on $\Theta(\overline{B}_R)$ which decompose in curves $S \mres \overline{B}_R$, $(T-S) \mres \overline{B}_R$ respectively (hence, $\eta_1+\eta_2$ decomposes $T\mres \overline{B}_R$) such that, on the ball $\overline{B}_R$,
$$\|S \| = \int_{\Theta(\overline{B}_R)} \|T_\gamma\| \de \eta_1, \quad \|T-S\|=\int_{\Theta(\overline{B}_R)} \|T_\gamma\| \de \eta_2.$$
By \cref{no-canc}, it also holds
$$\|S\|_\delta= \int_{\Theta(\overline{B}_R)} \|T_\gamma\|_\delta \de \eta_1, \quad \|T-S\|_\delta=\int_{\Theta(\overline{B}_R)} \|T_\gamma\|_\delta \de \eta_2.$$
on $\overline{B}_R$. Putting this together with (\ref{S,T delta}) we obtain on $\overline{B}_R$ the identity of measures
$$\|T \|_\delta =\int_{\Theta(\overline{B}_R)} \|T_\gamma\|_\delta \de(\eta_1+\eta_2).$$
Thus, applying another time \cref{no-canc},
\begin{equation*}\|T\|(\overline{B}_R)=\int_{\Theta(\overline{B}_R)} \|T_\gamma\|(\overline{B}_R) \de(\eta_1+\eta_2)=\|S\|(\overline{B}_R)+\|T-S\|(\overline{B}_R). \qedhere \end{equation*}
\end{proof}

\subsection{An auxiliary topology on open ended curves with locally finite length in $E$}\label{sectioncurves}

Let again $(E,d)$ be any complete metric space, as in \cref{MT}. We equip the set of curves $\Gamma(E)$ with the weak topology of local currents. That is, we define $\tau_\Gamma$ as the smallest topology on $\Gamma(E)$ such that the functions
$$\beta \in \Gamma(E) \mapsto T_\beta(f \de \pi)$$
are continuous, where $(f, \pi) \in \operatorname{Lip}_{\mathrm{bs}}(E) \times \operatorname{Lip}(E)$. \\
A fundamental system of neighbourhoods of $\bar\beta\in\Gamma(E)$ is given by
finite intersection of sets of the form
\begin{equation}
\left\{\beta\in\Gamma(E)\mid \bigl|T_\beta(f\,\de\pi)-T_{\bar\beta}(f\,\de\pi)\bigr|<\lambda\right\}
\end{equation} 
for $\lambda>0$ and $f\in{\rm Lip_{bs}}(E)$, $\pi\in {\rm Lip}(E)$.

\begin{remark}\label{meas-remark}
By \cref{rem_dualmass}, the maps $\beta \mapsto \|T_\beta\|(A)$, $\beta \mapsto\| \partial T_\beta\|(A)$ are lower semicontinuous on $(\Gamma(E), \tau_\Gamma)$ for any open set $A \subseteq E$. Then, by a straightforward application of Dynkin's $\pi$-$\lambda$ theorem one deduces that in general these maps are Borel-measurable for $A \subseteq E$ Borel, since the family
$$\mathcal{D}_R:=\{A \in \mathcal{B}(B_R(\bar{x})): \beta \mapsto \|T_\beta\|(A), \, \beta \mapsto \| \partial T_\beta\|(A) \mbox{ are Borel} \}$$
is a Dynkin system, for all $R>0$ (where we have fixed $\bar{x} \in E$).
\end{remark}

We recall the following simple result linking the Paolini-Stepanov topology on compact Lipschitz curves (see \cref{defTheta}) with the weak convergence of currents (see \cite[Lemma 4.1]{PS1}).

\begin{lemma}\label{PSlemma}
Let $(\gamma_n)_n \subseteq \Theta(E)$ be a sequence of curves with $\sup_n \ell(\gamma_n) < \infty$. If $\gamma_n \to \gamma$ in $\Theta(E)$ for some $\gamma \in \Theta(E)$, then $T_{\gamma_n} \to T_\gamma$
weakly.
\end{lemma}

\subsection{Conclusion} We are now ready to prove our main Theorem.
\begin{proof}[Proof of \cref{MT}]
It is not restrictive to work in the case $\mathbb{M}(\partial T)< \infty$, since our result is valid for general $T \in \mathbf{N}_{1,b}(E)$ once it holds for the currents $(S_i)_i$ which decompose $T$ as in \cref{infinite-boundary} (see also \cref{remark-reduction}). 
Also, as discussed at the beginning of \cref{riduzioneallinfinito}, we may assume $E=\ell^\infty$; this will allow us to apply our machinery. First, the decomposition $T=A+C$ with $ A,\, C \in \mathbf{N}_{1,b}(E)$, $C \leq T$, $\partial C=0$ and $A$ acyclic is given by \cref{decompcycacyc}.

By \cref{cor}, we can find a positive finite measure $\overline{\eta}$ in $\Theta(\overline{E})$, supported on $\{\gamma: \mathbb{M}_\delta(T_\gamma)=\ell_\delta(\gamma)\}$, such that
\begin{align}
&\overline{T}=\int_{\Theta(\overline{E})} T_\gamma \de\overline{\eta}(\gamma), \label{decomp2} \\
&\|\overline{T}\|_\delta=\int_{\Theta(\overline{E})} \|T_\gamma\|_\delta\de\overline{\eta}(\gamma), \label{masseq2}
\end{align}
\noindent and, moreover, if $T=A$ is acyclic (since then also $\overline{T}$ is acyclic by \cref{acyc-lemma}),
\begin{equation}
 \|\partial \overline{T}\|_\delta=\int_{\Theta(\overline{E})} \|\partial T_\gamma\|_\delta \de\overline{\eta}(\gamma), \label{boundaryeq2}   
\end{equation}
and $\overline{\eta}$-a.e. $\gamma \in \Theta(\overline{E})$ is an arc. 

\noindent
In the sequel, an open ended subcurve $\beta$ of $\gamma\in\Theta(\overline{E})$ is said to be a $x_\infty$-subcurve if 
$\beta:(a,b)\to E$ is a subcurve of $\gamma$ and $(a,b)$ is a connected component of $\{t\in (0,1):\ \gamma(t)\neq x_\infty\}$.

Next, for $F\subseteq\Gamma(E)$ we define the application $p_E(\gamma,F):\Theta(\overline{E})\to\N\cup\{\infty\}$ as
\begin{equation}\label{defp_E}
\begin{split}
p_E(\gamma,F):=\sum_{\beta\in F}n_\gamma(\beta),
\end{split}
\end{equation}
where $n_\gamma(\beta)$ is the number of times a $x_\infty$-subcurve $\beta\in\Gamma(E)$ occurs as a subcurve of $\gamma$
(in particular when $\gamma((0,1))\subseteq E$ we have that $n_\gamma(\beta)>0$ iff $\beta=\gamma\vert_{(0,1)}$).

 Notice that for $\overline{\eta}$-a.e. $\gamma$ all the subcurves $\beta$ have finite $\delta$-length
because $\gamma$ has finite length in $(\overline{E}, \delta)$ (in particular $\gamma$ and its subcurves $\beta$ have locally finite $d$-length in $E$ thanks to \cref{length-formula} applied to the elements of $\gamma\mres U$, for $U$ open bounded). In addition, they all lie in $\operatorname{supp}(T)$ for $\overline{\eta}$-a.e. $\gamma$ (as $\operatorname{supp}(\overline{T}) \subseteq \operatorname{supp}(T) \, \cup \{x_\infty\}$) and they are all arcs 
if $\gamma$ is an arc.

We use $p_E$ to define a kind of push-forward measure $\eta:=(p_E)_\#\overline{\eta}$, that is
\begin{equation}\label{pushf}
\eta(F):=\int_{\Theta(\overline{E})} p_E(\gamma,F) \de\overline{\eta}(\gamma)
\end{equation}
for a generic Borel set $F \subseteq \Gamma(E)$. Thanks to \cref{measurability} below, definition \eqref{pushf} is well-posed. From the definition, one can easily infer the change-of-variables formula
\begin{equation}\label{changeofv}
 \int_{\Gamma(E)} h (\beta)\de\eta(\beta)=\int_{\Theta(\overline{E})} \sum_{\{\beta:\ n_\gamma(\beta)>0\}} n_\gamma(\beta) h(\beta) \de\overline{\eta}(\gamma)
\end{equation}
for any $h \in L^1(\Gamma(E), \eta)$ or $h$ nonnegative and Borel. \\
Choosing $h(\beta)=\|T_\beta\|_\delta(B_R)$ in (\ref{changeofv}) (see \cref{meas-remark}) and using (\ref{masseq2}) one obtains
\begin{equation}\label{delta-dec}
\int_{\Gamma(E)} \|T_\beta\|_\delta(B_R) \de\eta(\beta)=\|T\|_\delta(B_R),
\end{equation}
where we have also exploited that for $\overline{\eta}$-a.e. $\gamma$ it holds $\mathbb{M}_\delta(T_\gamma)=\ell_\delta(\gamma)$, and thus
$$\|T_\gamma\|_\delta=\sum_{\{\beta:\ n_\gamma(\beta)>0\}} n_\gamma(\beta) \|T_\beta\|_\delta.$$

\noindent In particular, if we choose
$$h(\beta)=T_\beta(\omega), \quad\omega= (f, \pi)\in\operatorname{Lip}_{\mathrm{bs}}(E) \times\operatorname{Lip}(E), \, \operatorname{supp}(f) \subseteq B_R$$
then $h \in L^1(\Gamma(E), \eta)$ (as $|h(\beta)| \leq \|T_\beta\|_\delta \, \operatorname{Lip}^\delta(\pi \mres B_R)
\sup|f|$). Therefore, we can apply (\ref{changeofv}) to $h$ getting exactly (\ref{decomp}).  \\

\noindent Then, from \cref{no-canc} we obtain that (\ref{delta-dec}) also implies
$$\int_{\Gamma(E)} \|T_\beta\|(B_R) \de\eta(\beta)=\|T\|(B_R),$$
and thus, since $R>0$ is arbitrary, the inequality of measures $\| T\| \leq \int_{\Gamma(E)} \|T_\beta\| \de\eta(\beta)$ must be an equality, proving (\ref{masseq}). \\

\noindent Finally, for $T$ acyclic, choosing $h(\alpha)=\| \partial T_\alpha\|(B_R)$, one immediately obtains (\ref{boundaryeq}) from (\ref{boundaryeq2}).
Indeed (recall \cref{0-mass unchanged} and note that the boundary points in $E$ of an arc $\gamma \in \Theta(\overline{E})$ are exactly the union of the boundary points of the curves $\beta\in\gamma \mres E$, since $\gamma$ is only cut at infinity)         
\begin{equation*}
\begin{split}
\int_{\Gamma(E)} \|\partial T_\beta\|(B_R) \de\eta(\beta)&=\int_{\Theta(\overline{E})} 
\sum_{\{\beta:\ n_\gamma(\beta)>0\}} n_\gamma(\beta)\|\partial T_{\beta}\|(B_R)\de\overline{\eta}(\gamma) \\
&=\int_{\Theta(\overline{E})} \|\partial T_\gamma\|(B_R) \de\overline{\eta}(\gamma)\\
&=\| \partial \overline{T}\|(B_R)=\|\partial T\|(B_R). \qedhere
\end{split}
\end{equation*}
\end{proof}

\begin{remark}\label{transport-remark}
    The new results allow us to recover the ``transport'' part of the original statement between $(\partial T)^{-}$ and $(\partial T)^{+}$, with the added possibility of receiving or sending mass from infinity.
    In fact, if $A$ is the acyclic component of $T$ given by \cref{decompcycacyc}, the associated decomposition measure $\eta_{A}$ can be expressed as a sum of mutually singular measures $\eta_{B}, \eta_{BL}$ and $\eta_{BR}$, which are supported on the families of curves which are bounded, bounded on the left and bounded on the right  respectively.
    Then, by \cref{DefGamma} we note that $e_{0\,\#}\eta_{B}, e_{0\,\#}\eta_{BL}, e_{1\,\#}\eta_{B}$ and $e_{1\,\#}\eta_{BR}$ are well defined, where we denoted by
$$e_0:\{\beta \in \Gamma(E) \,\mid\,\operatorname{lim}_{t\to 0}d(\beta(t), \bar x)<\infty\}\to E,$$
$$e_{1}:\{\beta \in \Gamma(E)\,\mid\,\operatorname{lim}_{t\to1}d(\beta(t), \bar x)<\infty\}\to E$$
the maps $e_0(\beta):=\operatorname{lim}_{t\to 0}\beta(t)$ and $e_{1}(\beta):=\operatorname{lim}_{t\to 1}\beta(t)$.
    
Finally, we see from the precedent proof that we have 
 $$e_{0\,\#}\eta_{B}+ e_{0\,\#}\eta_{BL}=(\partial T)^{-}\quad\text{and}\quad e_{1\,\#}\eta_{B}+e_{1\,\#}\eta_{BR}=(\partial T)^{+}. $$
\end{remark}

\begin{remark}
    We remark that if $C$ is the cyclic component of $T$ in \cref{MT}, $\eta_C$ is the associated decomposition measure and $\eta^C_{B}, \eta^C_{BL}$ and $\eta^C_{BR}$ are  defined analogously to \cref{transport-remark}, then, by \cref{PS} and the construction argument in the proof of \cref{MT}, we find that 
    $$e_{0\,\#}\eta^C_{B}+ e_{0\,\#}\eta^C_{BL}= e_{1\,\#}\eta^C_{B}+e_{1\,\#}\eta^C_{BR}=\|C\|_{\delta} \leq \|T\|. $$
\end{remark}

In the proof above, we have assumed the following technical Lemma.

\begin{lemma}\label{measurability}
The function 
$p_E(\gamma,F)$ is Borel-measurable on $\Theta(\overline{E})$, for any Borel set $F$ in $(\Gamma(E), \tau_\Gamma)$.
\end{lemma}

\begin{proof}
We
prove that for all $L>0$ the restriction of $p_E(\gamma,F)$ to the Borel sets
$$\Theta_L:=\bigl\{\gamma\in\Theta(\overline{E})\mid \ell_\delta(\gamma)\leq L\bigr\}$$ 
is Borel measurable.
Fixed $\eps>0$, we denote by $\phi_F^\eps(\gamma)$ the number of maximal $x_\infty$-subcurves $\beta$ of $\gamma\in\Theta_L$, counted with multiplicity, 
which belong to $F$ and satisfy $\|\beta-x_\infty\|_\infty\geq\eps$, where
$$\| \beta-x_\infty\|_\infty:=\sup_{t \in (0,1)} \delta(\beta(t), x_\infty).$$
Notice that $\phi^\eps_F(\gamma)\leq\ell_\delta(\gamma)/\eps\leq L/\eps<\infty$, and that 
$\phi_F^\eps(\gamma)\uparrow p_E(\gamma,F)$ as $\eps\to 0$.
So, it suffices to check the measurability of the maps $\phi_F^\eps$, $F \in \mathcal{B}(\Gamma(E))$. \\
First, we prove that $\phi^\eps(\gamma):=\phi_{\Gamma(E)}^\eps(\gamma)$
is upper semicontinuous on $\Theta_L$. That is, if we take $\gamma_n \rightarrow \gamma$ in $\Theta_L$ (where we identify $\gamma_n$, $\gamma$ with their parametrizations on $[0,1]$ which give the uniform convergence) such that (up to a subsequence) $\phi^\eps(\gamma_n) \geq k$ for some $k \in \N$, then also $\phi^\eps(\gamma) \geq k$. Indeed, we can find $x_\infty$-subcurves $\beta_1^n, \ldots, \beta_k^n$ in $\Gamma(E)$ such that  
\begin{enumerate}
\setlength\itemsep{0.5em}
\item[(i)] $\beta_i^n$ is the restriction of $\gamma_n$ to some sub-interval $(a_i^n, b_i^n)\subseteq [0,1]$, with $b_i^n \leq a_{i+1}^n$ for any $1 \leq i < k$; 
\item[(ii)] $\| \beta_i^n-x_\infty\|_\infty \geq \eps$, for all $i=1, \ldots, k$. 
\end{enumerate}
The equicontinuity of $\gamma_n$, together with the condition $\| \beta_i^n-x_\infty\|_\infty \geq \eps$, provides a uniform lower bound on
$b_i^n-a_i^n$, hence one can pass to the limit as $n\to\infty$, up to subsequences, to obtain that $a_i^n \rightarrow a_i$, $b_i^n \rightarrow b_i$ 
for some $a_i, \, b_i \in [0,1]$ with $a_i<b_i$.  We deduce that affine reparametrizations of $\beta_i^n$ (closer and closer to the identity, recall that $\gamma_n \rightarrow \gamma$ uniformly) are uniformly convergent to $\beta_i= \gamma |_{(a_i, b_i)}$ with either
$\gamma(a_i)=x_\infty$ or $\gamma(b_i)=x_\infty$, or both (at least when $k>1$) and
$\| \beta_i-x_\infty\|_\infty\geq \eps$. Notice that
$\beta_i$ are not necessarily $x_\infty$-subcurves, because they are not necessarily $E$-valued. Nevertheless any
$\beta_i$ has a subcurve $\beta_i^*$ (and then a subcurve of $\gamma$) which is a $x_\infty$-subcurve with $\|\beta_i^*-x_\infty\|_\infty \geq \eps$, and this proves that $\phi^\eps(\gamma)\geq k$.
 
Now, we want to prove that $\phi^\eps_F$ is Borel measurable on $\Theta_L$ for any $F \subseteq E$ open. Note that it suffices to show it for $F \in \mathcal{E}_R$, for all $R \in \N$, where we define $\mathcal{E}_R\subseteq\Gamma(E)$ as the class made by unions of finite intersections of the form
\begin{equation}\label{finifini}
\bigcap_{j=1}^N\left\{\beta\in\Gamma(E)\mid \bigl|T_\beta(f_j\,\de\pi_j)-T_{\bar\beta_j}(f_j\,\de\pi_j)\bigr|<\lambda_j\right\}
\end{equation} 
for $\bar\beta_j\in\Gamma(E)$, $\lambda_j>0$, $f_j\in {\rm Lip_{bs}}(E)$, $\pi_j\in {\rm Lip}(E)$, $j=1,\ldots,N$, with
${\rm supp}(f_j)\subseteq B_R$. Indeed, since the sets of the form (\ref{finifini}) are a basis for the topology $\tau_\Gamma$ (letting also $R$ variable), any open set can be written as an increasing union $F=\bigcup_{R \in \N} F_R$, $F_R \subseteq F_{R+1}$, with $F_R \in \mathcal{E}_R$, and hence $\phi^\eps_F=\sup_R \phi^\eps_{F_R}$.

Consider then $F \in \mathcal{E}_R$, for some $R>0$. Let also $\eps_0:=\delta(x_\infty, \overline{B}_R)$, so that if some curve $\beta \in \Gamma(E)$ satisfies $\|\beta-x_\infty\|_\infty < \eps_0$ then $\beta$ does not intersect the ball $B_R$. 
We show that the function $\phi_{F}^\eps(\gamma)$ is lower semicontinuous when restricted to any of the Borel sets $A_{k,h} \cap \Theta_L$, where
$$A_{k,h}:=\{\gamma \in \Theta(\overline{E}): \, \phi^\eps(\gamma)=k, \, \phi^{\eps_0}(\gamma)=h\},$$
for $k, \, h \in \N$ ($A_{k,h}$ is Borel thanks to the measurability of $\phi^\eps, \, \phi^{\eps_0}$). This is sufficient to prove the measurability of $\phi_F^\eps$: indeed, the finiteness of $\phi^\eps$ gives
$$\Theta(\overline{E})=\bigcup_h \bigcup_k A_{k,h}.$$
Now, take $\gamma_n \rightarrow \gamma$ in $\Theta(\overline{E})$, with $\gamma_n, \gamma \in A_{k,h} \cap \Theta_L$, and find
$x_\infty$-subcurves $\beta_1^n, \ldots, \beta_k^n$ of $\gamma_n$ as above. Note that the limits $\beta_1, \ldots, \beta_k$ built as in the previous step
may not be $E$-valued and hence $\beta_i \, \mres \, E$ may consist of multiple curves. However, for any $i=1,\ldots,k$, exactly one of the $x_\infty$-subcurves of $\beta_i$, 
which we will call $\beta_i^*$ (possibly coinciding with $\beta_i$ itself), satisfies $\|\beta_i^*-x_\infty\|_\infty \geq \eps$, since otherwise we would have $\phi^\eps(\gamma)>k$. Also, we claim that $\phi^{\eps_0}(\gamma)=h$ implies that $\beta_i^*$ is the only element of $\beta_i \mres E$ such that $\|\beta_i^*-x_\infty\|_\infty \geq \eps_0$. This is obvious for $\eps_0\geq\eps$, while for $\eps_0<\eps$ it can be proved constructing in the same way 
$x_\infty$-subcurves $\widetilde{\beta}_1^n, \ldots, \widetilde{\beta}_h^n$ of $\gamma_n$, 
with $\| \widetilde{\beta}_i^n-x_\infty\|_\infty \geq \eps_0$; exactly $k$ of them will coincide with $\beta_1^n, \ldots, \beta_k^n$. 
This leads to the inclusion 
$$\{\beta_1,\ldots,\beta_k\}\subseteq\{\tilde{\beta}_1,\ldots,\tilde{\beta}_h\}$$
for the limit curves, with $\tilde{\beta}_i$ distinct, each one containing a $x_\infty$-subcurve with distance from $x_\infty$ larger than 
$\eps_0$. Hence, if some $\beta_i$ were to contain more than one $x_\infty$-subcurve $\alpha$ with
$\|\alpha-x_\infty\|_\infty\geq\eps_0$, we would contradict $\phi^{\eps_0}(\gamma)=h$. Calling $\bar\beta_i$ that subcurve,
the inequality $\eps_0<\eps$ gives $\bar\beta_i=\beta_i^*$, proving the claim.
Thus, our choice
of $\eps_0$ gives
$$\beta_i \mres B_R=\beta_i^* \mres B_R,$$
and therefore the convergence $\beta_i^n \rightarrow \beta_i$ (thought as curves in $\Theta(\overline{E})$) combined with \cref{PSlemma} (here we use that $\gamma_n \in \Theta_L$ to infer that $\ell_\delta(\beta_i^n)$ are uniformly bounded by $L$) yields
$$
\lim_{n\to\infty}T_{\beta^n_i}(f \de\pi)=T_{\beta_i}(f \de\pi)=T_{\beta_i^*}(f\de\pi) \qquad \forall i=1, \ldots, k,
$$
for all $f \in {\rm Lip_{bs}}(E)$, $\pi\in {\rm Lip}(E)$ with
${\rm supp}(f)\subseteq B_R$.
In particular, if $\beta_i^* \in F$ for some $1 \leq i \leq k$ then $\beta_i^*$ belongs to one of the sets in \eqref{finifini}
and so also $\beta_i^n \mres E \in F$ for $n$ sufficiently large.
Hence, $\liminf_n \phi_F^\eps(\gamma_n) \geq \phi_F^\eps(\gamma)$, proving the desired semicontinuity. \\
To conclude, we obtain that $\phi_F^\eps$ is Borel-measurable on $\Theta_L$ for any Borel set $F$ noticing that the family
$$\mathcal{M}:=\bigl\{F \in \mathcal{B}(\Gamma(E)): \phi_F^\eps \mbox{ is Borel-measurable on  }\Theta_L \bigr\}$$
is a Dynkin system containing the $\pi$-system of open sets. Indeed, if $F \in \mathcal{M}$ then also $\Gamma(E)\setminus F\in \mathcal{M}$, because
$\phi^\eps-\phi_F^\eps$ 
is Borel-measurable, since we have already checked that $\Gamma(E) \in \mathcal{M}$, and the difference makes sense as $\phi^\eps$ only attains finite values. Furthermore, the implication $F_n \in \mathcal{M}$, $F_n$ disjoint $\Rightarrow \, \bigcup_n F_n \in \mathcal{M}$ is trivial. 
Thus, by Dynkin's $\pi$-$\lambda$ Theorem, $\mathcal M$ contains the whole Borel $\sigma$-algebra. Eventually we let $\eps \to 0$ and $L \to\infty$.
\end{proof}

\newpage
\printbibliography

@article {AT2,
    AUTHOR = {Ambrosio, Luigi and Trevisan, Dario},
     TITLE = {Well-posedness of {L}agrangian flows and continuity equations
              in metric measure spaces},
   JOURNAL = {Anal. PDE},
  FJOURNAL = {Analysis \& PDE},
    VOLUME = {7},
      YEAR = {2014},
    NUMBER = {5},
     PAGES = {1179--1234},
}

@article {ST,
    AUTHOR = {Stepanov, Eugene and Trevisan, Dario},
     TITLE = {Three superposition principles: currents, continuity equations
              and curves of measures},
   JOURNAL = {J. Funct. Anal.},
  FJOURNAL = {Journal of Functional Analysis},
    VOLUME = {272},
      YEAR = {2017},
    NUMBER = {3},
     PAGES = {1044--1103},
}

@article {AK0,
    AUTHOR = {Ambrosio, Luigi and Kirchheim, Bernd},
     TITLE = {Rectifiable sets in metric and {B}anach spaces},
   JOURNAL = {Math. Ann.},
  FJOURNAL = {Mathematische Annalen},
    VOLUME = {318},
      YEAR = {2000},
    NUMBER = {3},
     PAGES = {527--555},
      ISSN = {},
   MRCLASS = {28A75 (46G99 46T99 49Q20)},
  MRNUMBER = {1800768},
MRREVIEWER = {Piotr\ Haj\l asz},
       DOI = {10.1007/s002080000122},
       URL = {},
}

@book {AT,
    AUTHOR = {Ambrosio, Luigi and Tilli, Paolo},
     TITLE = {Topics on analysis in metric spaces},
    SERIES = {Oxford Lecture Series in Mathematics and its Applications},
    VOLUME = {25},
 PUBLISHER = {Oxford University Press, Oxford},
      YEAR = {2004},
     PAGES = {},
      ISBN = {},
   MRCLASS = {},
  MRNUMBER = {},
MRREVIEWER = {},
}

@article {AK,
    AUTHOR = {Ambrosio, Luigi and Kirchheim, Bernd},
     TITLE = {Currents in metric spaces},
   JOURNAL = {Acta Math.},
  FJOURNAL = {Acta Mathematica},
    VOLUME = {185},
      YEAR = {2000},
    NUMBER = {1},
   MRCLASS = {49Q15 (49Q20)},
  MRNUMBER = {1794185},
MRREVIEWER = {Giovanni\ Bellettini},
}

@book {ABS,
    AUTHOR = {Ambrosio, Luigi and Bru\'e, Elia and Semola, Daniele},
     TITLE = {Lectures on optimal transport},
    SERIES = {Unitext},
    VOLUME = {130},
      NOTE = {La Matematica per il 3+2},
 PUBLISHER = {Springer, Cham},
      YEAR = {2021},
   MRCLASS = {49-01 (49Q22)},
  MRNUMBER = {4294651},
MRREVIEWER = {Hugo\ Lavenant},
}

@misc{AILP2,
    AUTHOR = {Ambrosio, Luigi and Ikonen, Toni and Lu\v{c}i\'{c}, Danka and
              Pasqualetto, Enrico},
     TITLE = {\normalfont{Metric {S}obolev spaces {II}: Dual energies and divergence measures}},
    PUBSTATE={In preparation},
}

@book {B,
    AUTHOR = {Bogachev, V. I.},
     TITLE = {Measure theory. {V}ol. {I}, {II}},
 PUBLISHER = {Springer-Verlag, Berlin},
      YEAR = {2007},
   MRCLASS = {28-02 (28Axx 28Cxx 46G12 60G42 60G44)},
  MRNUMBER = {2267655},
MRREVIEWER = {Ren\'e\ L.\ Schilling},
}

@book {F,
    AUTHOR = {Federer, Herbert},
     TITLE = {Geometric measure theory},
    SERIES = {Die Grundlehren der mathematischen Wissenschaften},
    VOLUME = {Band 153},
 PUBLISHER = {Springer-Verlag New York, Inc., New York},
      YEAR = {1969},
   MRCLASS = {28.80 (26.00)},
  MRNUMBER = {257325},
MRREVIEWER = {J.\ E.\ Brothers},
}

@article{FR,
author = {Fr{\'e}chet, M.},
journal = {Mathematische Annalen},
pages = {145-168},
title = {Les dimensions d'un ensemble abstrait},
volume = {68},
year = {1910},
}

@article {UL1,
    AUTHOR = {Lang, Urs},
     TITLE = {Local currents in metric spaces},
   JOURNAL = {J. Geom. Anal.},
  FJOURNAL = {Journal of Geometric Analysis},
    VOLUME = {21},
      YEAR = {2011},
    NUMBER = {3},
     PAGES = {683--742},
   MRCLASS = {49Q15},
  MRNUMBER = {2810849},
}

@article {UL2,
    AUTHOR = {Lang, Urs and Wenger, Stefan},
     TITLE = {The pointed flat compactness theorem for locally integral
              currents},
   JOURNAL = {Comm. Anal. Geom.},
  FJOURNAL = {Communications in Analysis and Geometry},
    VOLUME = {19},
      YEAR = {2011},
    NUMBER = {1},
     PAGES = {159--189},
   MRCLASS = {58A25 (49Q15 53C65)},
  MRNUMBER = {2818408},
MRREVIEWER = {Luca\ Granieri},
}

@article {PS1,
    AUTHOR = {Paolini, Emanuele and Stepanov, Eugene},
     TITLE = {Decomposition of acyclic normal currents in a metric space},
   JOURNAL = {J. Funct. Anal.},
  FJOURNAL = {Journal of Functional Analysis},
    VOLUME = {263},
      YEAR = {2012},
    NUMBER = {11},
     PAGES = {3358--3390},
   MRCLASS = {49Q05 (54E35)},
  MRNUMBER = {2984069},
}

@article {PS2,
    AUTHOR = {Paolini, Emanuele and Stepanov, Eugene},
     TITLE = {Structure of metric cycles and normal one-dimensional
              currents},
   JOURNAL = {J. Funct. Anal.},
  FJOURNAL = {Journal of Functional Analysis},
    VOLUME = {264},
      YEAR = {2013},
    NUMBER = {6},
     PAGES = {1269--1295},
   MRCLASS = {49Q15 (58A25)},
  MRNUMBER = {3017264},
MRREVIEWER = {Laurent\ Moonens},
}

@article {SM,
    AUTHOR = {Smirnov, S. K.},
     TITLE = {Decomposition of solenoidal vector charges into elementary
              solenoids, and the structure of normal one-dimensional flows},
   JOURNAL = {Algebra i Analiz},
  FJOURNAL = {Rossi\u{\i}skaya Akademiya Nauk. Algebra i Analiz},
    VOLUME = {5},
      YEAR = {1993},
    NUMBER = {4},
     PAGES = {206--238},
  MRNUMBER = {1246427},
MRREVIEWER = {Andrew Bucki},
}
\end{document}